\documentclass[reqno]{amsart}
\usepackage{amsmath, amsfonts, amssymb, amsthm, amscd}
\usepackage{hyperref}

\newtheorem{theorem}{Theorem}[section]
\newtheorem{lemma}[theorem]{Lemma}
\newtheorem{corollary}[theorem]{Corollary} 

\newtheorem{proposition}[theorem]{Proposition}

\theoremstyle{definition}
\newtheorem{definition}[theorem]{{Definition}}
\newtheorem{example}[theorem]{Example}
\newtheorem{remark}[theorem]{Remark}
\newtheorem{notation}[theorem]{Notation}
\newtheorem{chunk}[theorem]{\hspace*{-.77ex}\bf}
  
\newtheorem{fact}[theorem]{\bf Fact}
\newtheorem{para}[theorem]{}
\newtheorem{Question}[theorem]{Question}
\newtheorem{Problem}[theorem]{Problem}

\newtheorem*{chunk*}{}

\numberwithin{equation}{theorem}

\usepackage[all]{xy}
\usepackage{enumerate}

\usepackage{tikz}
\usetikzlibrary{shapes}
\usetikzlibrary{snakes}

\usepackage{graphicx}

\usepackage{pgf}
\usepackage{latexsym,amsmath,amssymb}
\usepackage[parfill]{parskip} 

\newif\ifdviwin

\usepackage[latin1]{inputenc}
\usepackage[english]{babel}
\usepackage{indentfirst}
\usepackage[mathscr]{eucal}
\usepackage{amsmath,amscd}
\usepackage{xxcolor}

\def\p{{\mathfrak p}}
\def\q{{\mathfrak q}}

\def\m{{\mathfrak m}}
\def\H{\operatorname{H}}
\def\Hom{\operatorname{Hom}}

\def\Coker{\operatorname{Coker}}

\newcommand{\Spec}{\operatorname{Spec}}

\def\dim{\operatorname{dim}}

\def\ann{\operatorname{ann}}

\def\Min{\operatorname{Min}}

\def\Coker{\operatorname{Coker}}

\def\p{\mathfrak{p}}
\def\sub{\text{-}}

\newcommand{\ol}{\overline}
\newcommand{\Ht}{\operatorname{ht}}

\newcommand{\fp}{\mathfrak p}
\newcommand{\fq}{\mathfrak q}

\newcommand{\fs}{\mathfrak s}

\newcommand{\fr}{\mathfrak r}

\newcommand{\into}{\hookrightarrow}

\newcommand{\mspec}{m \sub \Spec}
\newcommand{\summands}{\zeta}
\newcommand{\components}{\beta}
\newcommand{\disjointunion}{\bigsqcup}

\begin{document}
\bibliographystyle{amsplain}



\title
{On the structure of $S_2$-ifications of complete local rings}

\author[S.\ Sather-Wagstaff]{Sean Sather-Wagstaff}
\address{Sean Sather-Wagstaff, Department of Mathematics,
300 Minard Hall,
North Dakota State University,
Fargo, North Dakota 58105-5075, USA}
\email{Sean.Sather-Wagstaff@ndsu.edu}
\urladdr{http://www.ndsu.edu/pubweb/\~{}ssatherw/}

\thanks{
Sean Sather-Wagstaff was supported in part by a grant from the NSA.
Sandra Spiroff was supported in part by a grant from the Simons Foundation}

\author{Sandra Spiroff}
\address{Sandra Spiroff, Department of Mathematics, 
305 Hume Hall,
P.~O.~Box 1848,
University of Mississippi, University, MS  38677-1848, USA}
\email{spiroff@olemiss.edu}
\urladdr{http://home.olemiss.edu/\~{ }spiroff/}

\begin{abstract}  Motivated by work of  Hochster and Huneke, we investigate 
several constructions related to the $S_2$-ification $T$ of a complete equidimensional local ring $R$:
the canonical module, the top local cohomology module, topological spaces of the form $\Spec(R)-V(J)$,
and the (finite simple) graph $\Gamma_R$ with vertex set $\Min(R)$ defined by Hochster and Huneke.
We generalize one of their results by showing, e.g., that the number of maximal ideals of $T$ is equal to the
number of connected components of $\Gamma_R$. We further investigate this graph by exhibiting
a technique for showing that a given graph $G$ can be realized as one of the form $\Gamma_R$.
\end{abstract}

\date \today
\maketitle




\section{Introduction}
\label{sec140103a}

\begin{chunk}
Throughout this paper, the term ``ring'' is short for ``noetherian ring'', and ``graph'' is short for ``finite simple undirected graph''.
In addition, $\mathsf k$ will be a field, and $(R,\m,\mathsf k)$ a local ring.
\end{chunk}

This project takes its motivation from a paper by M.~Hochster and C.~Huneke \cite{HH}, regarding $S_2$-ifications of complete, equidimensional, local rings, where by ``equidimensional" we mean that  $\dim (R/\fp) = \dim (R)$ for every minimal prime $\fp$ of $R$.  (See Section~\ref{sec140103f} for $S_2$-ification definitions and background material.)
Our interest in this subject comes from our paper~\cite{SWS2} where we use~\cite[(3.9)]{HH}
to show that a certain integral closure has to be local. 
The utility of this construction has led us to investigate its properties more carefully. 
In this paper, we focus on the following construction and subsequent result.

\begin{definition} \cite[(3.4)]{HH} \label{graph} Assume that $R$ is equidimensional.  We denote by $\Gamma_R$ the graph whose vertices are the minimal primes of $R$, and whose edges are determined by the following rule: if $\fp, \fq$ are distinct minimal primes of $R$, then $\fp$ and $\fq$ are adjacent in $\Gamma_R$ if and only if $\Ht_R(\fp+\fq)=1$.
\end{definition}

\begin{fact} \cite[(3.6)]{HH} \label{HH} If $R$ is complete and equidimensional, then the following conditions are equivalent:
\begin{enumerate}[(i)]
\item[(a)] The local cohomology module $H^{\dim (R)}_{\mathfrak m}(R)$ is indecomposable;
\item[(b)] The canonical module of $R$ is indecomposable;
\item[(c)] The $S_2$-ification of $R$ is local;
\item[(d)] For every ideal $J$ of height at least two, $\Spec(R) -V(J)$ is connected;
\item[(e)] The graph $\Gamma_R$ from Definition \ref{graph} is connected.
\end{enumerate}
\end{fact}

The first main result of the current paper is a generalization of this fact, which requires a bit of notation/discussion.

\begin{notation}\label{notn140107a}
Assume  that $R$ is   complete.
The Krull-Remak-Schmidt Theorem states that a finitely generated $R$-module decomposes uniquely
as a direct sum of indecomposable $R$-modules. 
By Matlis duality, the same is true for artinian $R$-modules. 
For an $R$-module $M$ that is either finitely generated or artinian, let 
$\summands_R(M)$ denote the number of summands in a direct sum decomposition of $M$ by indecomposable $R$-modules.
For a topological space or graph $X$, let
$\components(X)$ denote the number of connected components of $X$.
For a ring $S$, let $\mspec(S)$ denote the set of its maximal ideals. 
\end{notation}

Here is our generalization of Fact~\ref{HH}.
Its proof  is spread throughout  Section~\ref{sec140103c}; 
see~\ref{proof140106a}. 

\begin{theorem} \label{thm140103a} If $R$ is complete and equidimensional, then the following quantities are equal:
\begin{enumerate}[\ \rm(a)]
\item $\summands_R(\H^{\dim (R)}_{\mathfrak m}(R))$;
\item $\summands_R(\omega)$ where $\omega$ is a canonical module for $R$;
\item $|\mspec(T)|$ where $T$ is
the $S_2$-ification of $R$;
\item $\max\{\components(\Spec(R) -V(J))\mid \text{$J$ is an ideal of $R$ such that $\Ht_R(J)\geq 2$}\}$;
\item $\components(\Gamma_R)$.
\end{enumerate}
\end{theorem}

In the process of proving this result, we developed a certain interest in understanding more about the graph $\Gamma_R$.
This is the subject of Sections~\ref{sec140103b} and~\ref{sec140103e}.
In the first of these sections, 
we work to familiarize the reader with this construction via explicit computations and preliminary results.

In Section~\ref{sec140103e}, we investigate the following question: given a graph $G$, does there exist a 
complete local equidimensional ring $R$ such that $\Gamma_R$ is graph-isomorphic to $G$? In particular, we describe a 
labeling procedure for graphs, called an admissible labeling, that gives a large class of graphs where the answer is affirmative.
(See Definition~\ref{labels}.)

\begin{theorem} \label{thm140103b}
Let $G$ be a graph. If $G$ admits an admissible labeling, then there is 
a complete local equidimensional ring $R$ such that $\Gamma_R$ is graph-isomorphic to $G$.
Moreover, the ring $R$ is of the form $\mathsf k[\![X_1,\ldots,X_n]\!]/I$ where 
$I$ is generated by square-free monomials in the variables $X_1,\ldots,X_n$.
\end{theorem}

The proof of this result is contained in~\ref{proof140103a}.
We also show that certain standard classes of graphs (e.g., complete graphs, cycles, and paths) do have admissible labelings,
and we exhibit a graph on 5 vertices that does not admit an admissible labeling.

\section{Background}
\label{sec140103f}

\subsection*{Canonical modules and $S_2$-ifications}

\begin{definition} 
Let $E$ be the injective hull of $\mathsf k$ over $R$.
A \textit{canonical module} for $R$ is a finitely generated $R$-module $\omega$ such that
$\Hom_R(\omega,E)$ is isomorphic to the local cohomology module $\operatorname{H}^{\dim(R)}_{\m}(R)$.
\end{definition}

\begin{fact}
If $R$ is a homomorphic image of a Gorenstein local ring, e.g., $R$ is complete, then it has a canonical module.
\end{fact}

\begin{definition} \label{j-ideal} 
Denote by $j(R)$ the largest ideal which is a submodule of $R$ of dimension smaller than $\dim (R)$.  Specifically, 
$$j(R) = \{a \in R \mid \dim(R/ \ann_R(a)) < \dim (R) \}.$$
\end{definition}

\begin{fact}  \label{fact1}  
The (local) ring $R$ is equidimensional and unmixed (i.e., has no embedded associated primes) if and only if $j(R)= (0)$; see~\cite[(2.1)]{HH}.
In particular, if $R$ is a local domain, then $j(R)=0$.
\end{fact}

\begin{definition}  \label{S2def} \cite[(2.3)]{HH}  
\begin{enumerate}[(a)]
\item
If $j(R) = 0$, then an $R$-subalgebra $T$ of the total  ring of quotients of $R$ is an {\it $S_2$-ification} of $R$ if:
\begin{enumerate}
\item[$\bullet$] $T$ is module finite over $R$;
\item[$\bullet$] $T$ satisfies the Serre condition $(S_2)$ over $R$; and
\item[$\bullet$] $\Coker(R \to T)$ has no prime ideal of $R$ of height less than two in its support.  
\end{enumerate}
\item
When $R$ is equidimensional  but possibly $j(R) \neq 0$, then by an {\it $S_2$-ification} of $R$, we mean 
 an $S_2$-ification of $R/j(R).$
\end{enumerate}
\end{definition}

\begin{fact}  \label{fact2} \cite[(2.7)]{HH}  If $R$ has a canonical module $\omega$, then $R$ has an $S_2$-ification.
Specifically, $\Hom_R(\omega, \omega)$ is an $S_2$-ification of $R$.  
\end{fact}

\subsection*{Graphs}

We only use basic facts from graph theory; see, e.g., the text of Diestel~\cite{D}.

\begin{notation}
Let $n$ be a positive integer. 
The complete graph on $n$ vertices (i.e., the $n$-clique) is denoted $K_n$.
The path on $n$ vertices is denoted $P_{n-1}$. 
The cycle on $n \geq 3$ vertices is denoted $C_n$.
\end{notation}

\begin{definition}
A graph in which all the vertices have the same degree is \emph{regular}.
\end{definition}

\begin{definition}
Let $G$ be a connected graph with vertex set $V$.
A \emph{spanning tree} of $G$ is a tree $T$ that is a subgraph of $G$ with vertex set $V$.
\end{definition}

\begin{remark}
It is straightforward to show that every connected graph has a spanning tree.
\end{remark}


\section{An Introduction to $\Gamma_R$}
\label{sec140103b}


To get a feel for the graph $\Gamma_R$, this section consists of
explicit computations and preliminary results, some of which will be useful in Section~\ref{sec140103e}.
We begin with some small examples.

\begin{example} \label{dim1} If $|\Min(R)|=1$ (e.g., if $R$ is a domain) or $\dim(R)\leq 1$, then 
$\Gamma_R = K_{|\Min(R)|}$, so 
the equivalent conditions in Fact \ref{HH} are satisfied.  
Indeed, if $|\Min(R)|=1$, then $\Gamma_R$ is an isolated vertex.
If $\dim (R) \leq 1$, then for all $\mathfrak p, \mathfrak q \in \Min(R)$, we have $\Ht_R(\mathfrak p + \mathfrak q) \leq 1$.  
In each case, the desired conclusion follows immediately by definition.
\end{example}

\begin{example}  \label{ex140103b}
There are two graphs on two vertices, namely,  the path $P_1$ and the disjoint union of two  vertices.  
The two possibilities are realized as $\Gamma_R$
by the rings $R=\mathsf k[\![X_1,X_2]\!]/(X_1X_2)$ and $R=\mathsf k[\![X_1,X_2,X_3,X_4]\!]/(X_1X_2, X_2X_3, X_3X_4, X_1X_4)$, respectively, whose graphs are shown below.

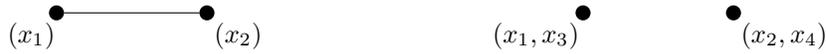
\begin{align*}
\begin{pgfpicture}{3cm}{-.5cm}{6cm}{0cm}
\pgfnodecircle{Node1}[fill]{\pgfxy(0,0)}{0.1cm}
\pgfnodecircle{Node2}[fill]{\pgfxy(2,0)}{0.1cm}
\pgfnodecircle{Node3}[fill]{\pgfxy(7, 0)}{0.1cm}
\pgfnodecircle{Node4}[fill]{\pgfxy(9, 0)}{0.1cm}
\pgfnodeconnline{Node1}{Node2}
\pgfputat{\pgfxy(-.65, -.3)}{\pgfbox[left,center]{$(x_1)$}}
\pgfputat{\pgfxy(2.1,-.3)}{\pgfbox[left,center]{$(x_2)$}}
\pgfputat{\pgfxy(5.8,-.3)}{\pgfbox[left,center]{$(x_1, x_3)$}}
\pgfputat{\pgfxy(9.1, -.3)}{\pgfbox[left,center]{$(x_2,x_4)$}}
\end{pgfpicture}
\end{align*}
\end{example}

\begin{example} \label{min=3} 
There are four graphs on three vertices, and each is realized as $\Gamma_R$ as pictured below.
Indeed, the first two (i.e., the connected ones) are associated to the rings $\mathsf k[\![X_1,X_2, X_3]\!]/(X_1X_2X_3)$ and $\mathsf k[\![X_1,X_2, X_3,X_4]\!]/(X_1X_2, X_2X_3, X_3X_4)$.  The two disconnected graphs arise from $\mathsf k[\![X_1, \dots, X_5]\!]/I$ and $\mathsf k[\![X_1, \dots, X_6]\!]/J$, where $I = (X_1, X_2) \bigcap  (X_3, X_4) \bigcap (X_3, X_5)$ and $J = (X_1, X_2) \bigcap (X_3, X_4) \bigcap (X_5, X_6)$.

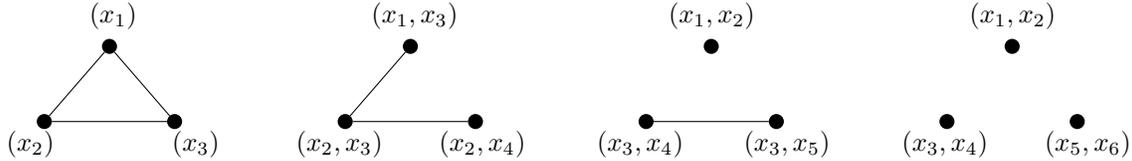
\begin{align*}
\begin{pgfpicture}{7cm}{-.5cm}{6cm}{1cm}
\pgfnodecircle{Node1}[fill]{\pgfxy(0,0)}{0.1cm}
\pgfnodecircle{Node2}[fill]{\pgfxy(.8666,1)}{0.1cm}
\pgfnodecircle{Node3}[fill]{\pgfxy(1.732, 0)}{0.1cm}
\pgfnodeconnline{Node1}{Node2}
\pgfnodeconnline{Node1}{Node3}
\pgfnodeconnline{Node2}{Node3}
\pgfputat{\pgfxy(-.5, -.3)}{\pgfbox[left,center]{$(x_2)$}}
\pgfputat{\pgfxy(.6,1.4)}{\pgfbox[left,center]{$(x_1)$}}
\pgfputat{\pgfxy(1.7,-.3)}{\pgfbox[left,center]{$(x_3)$}}
\pgfnodecircle{Node4}[fill]{\pgfxy(4,0)}{0.1cm}
\pgfnodecircle{Node5}[fill]{\pgfxy(4.8666,1)}{0.1cm}
\pgfnodecircle{Node6}[fill]{\pgfxy(5.732, 0)}{0.1cm}
\pgfnodeconnline{Node4}{Node5}
\pgfnodeconnline{Node4}{Node6}
\pgfputat{\pgfxy(3.35, -.3)}{\pgfbox[left,center]{$(x_2,x_3)$}}
\pgfputat{\pgfxy(4.35,1.4)}{\pgfbox[left,center]{$(x_1,x_3)$}}
\pgfputat{\pgfxy(5.25,-.3)}{\pgfbox[left,center]{$(x_2,x_4)$}}
\pgfnodecircle{Node7}[fill]{\pgfxy(8,0)}{0.1cm}
\pgfnodecircle{Node8}[fill]{\pgfxy(8.8666,1)}{0.1cm}
\pgfnodecircle{Node9}[fill]{\pgfxy(9.732, 0)}{0.1cm}
\pgfnodeconnline{Node7}{Node9}
\pgfputat{\pgfxy(7.35, -.3)}{\pgfbox[left,center]{$(x_3,x_4)$}}
\pgfputat{\pgfxy(8.3,1.4)}{\pgfbox[left,center]{$(x_1,x_2)$}}
\pgfputat{\pgfxy(9.3,-.3)}{\pgfbox[left,center]{$(x_3,x_5)$}}
\pgfnodecircle{Node10}[fill]{\pgfxy(12,0)}{0.1cm}
\pgfnodecircle{Node11}[fill]{\pgfxy(12.8666,1)}{0.1cm}
\pgfnodecircle{Node12}[fill]{\pgfxy(13.732, 0)}{0.1cm}
\pgfputat{\pgfxy(11.4, -.3)}{\pgfbox[left,center]{$(x_3,x_4)$}}
\pgfputat{\pgfxy(12.3,1.4)}{\pgfbox[left,center]{$(x_1,x_2)$}}
\pgfputat{\pgfxy(13.3,-.3)}{\pgfbox[left,center]{$(x_5,x_6)$}}
\end{pgfpicture}
\end{align*}
(These graphs are easily verified using Fact~\ref{fact140103a} below.)
\end{example}

Because many of our examples come from monomial ideals, we present some well known properties about them next.

\begin{fact}\label{fact140103a}
Let $R=\mathsf k[\![X_1,\ldots,X_n]\!]/I$ where $I$ is generated by monomials in the variables $X_1,\ldots,X_n$.
Then the minimal primes of $R$ are generated by sublists of the variables, that is,
they are of the form $(X_{i_1},\ldots,X_{i_m})R$ where $1\leq i_1<\cdots<i_m\leq n$.
The ring $R$ is equidimensional precisely when each minimal prime is generated by the same number of variables.
Assume that $R$ is equidimensional, and consider two minimal primes $\p=(X_{i_1},\ldots,X_{i_m})R$
and $\q=(X_{j_1},\ldots,X_{j_m})R$. Then $\Ht_R(\p+\q)=0$ (i.e., $\fp=\mathfrak q$) if and only if $\{i_1,\ldots,i_m\}=\{j_1,\ldots,j_m\}$.
And $\Ht_R(\p+\q)=1$ if and only if the sets $\{i_1,\ldots,i_m\}$ and $\{j_1,\ldots,j_m\}$ differ by exactly one element, that is, if and only if
$|\{i_1,\ldots,i_m\}\bigcap\{j_1,\ldots,j_m\}|=m-1$,
that is, if and only if $|\{i_1,\ldots,i_m\}\bigcup\{j_1,\ldots,j_m\}|=m+1$.
\end{fact}

In the remainder of the section, we show how to construct a ring $R$ such that $\Gamma_R$ takes a familiar form, for example, an arbitrary cycle.
In particular, there are rings $R$ whose graphs have arbitrarily large diameter and girth.  Likewise, there are rings $R$ whose graphs are complete or regular.  

\begin{example} \label{cycle} For any integer $n \geq 3$, the graph $\Gamma_R$ of the ring $R = \mathsf k[\![X_1, \dots, X_n]\!]/J$, where 
$$J = (X_1, X_2) \bigcap  (X_2, X_3) \bigcap \dots \bigcap (X_{n-1}, X_n)  \bigcap (X_n, X_1)$$
is the cycle $C_n$, which has girth $n$ and diameter $\lfloor \frac{n}{2} \rfloor$.  
\end{example}


\begin{proposition}  \label{complete} 
If $R$ is a complete hypersurface with $|\Min(R)|=n$, then $\Gamma_R=K_n$.
\end{proposition}

\begin{proof} 
By assumption, we have 
$R \cong Q/(f_1^{e_1}f_2^{e_2} \cdots f_n^{e_n})$ where $f_1, f_2, \dots, f_n$ are non-zero non-associate primes in a complete regular local ring $Q$  and each $e_i \geq 1$.  
Then each $(f_i)R$ represents a vertex in $\Gamma_R$, and $(f_i)R$ is adjacent to $(f_j)R$ if and only if $i \neq j$.
\end{proof}

\begin{example}\label{ex140103a}
Let $n$ be a positive integer. Then the ring $R=\mathsf k[\![X_1, \dots, X_n]\!]/(X_1\cdots X_n)$ satisfies the hypotheses of
Proposition~\ref{complete}, so $\Gamma_R=K_n$.
\end{example}

\begin{proposition}  If $R$ is a complete monomial complete intersection, then $\Gamma_R$ is  regular and connected.
\end{proposition}

\begin{proof} By assumption, we have
$$R = \frac{\mathsf k[\![X_{11}, \dots, X_{1i_1}, X_{21}, \dots, X_{2i_2}, \dots, X_{m1}, \dots, X_{mi_m}, Y_1, \dots, Y_r]\!]}{(X_{11}^{e_{11}} \cdots X_{1i_1}^{e_{1i_1}}, X_{21}^{e_{21}} \cdots X_{2i_2}^{e_{2i_2}}, \dots, X_{m1}^{e_{m1}} \cdots X_{mi_m}^{e_{mi_m}})},$$
where each $e_{jk} \geq 1$, each $i_k \geq 1$, and $r \geq 0$.  
The minimal primes of $R$ are  the ideals of the form $\mathfrak p = (X_{1j_1}, X_{2j_2}, \dots, X_{mj_m})R$, so $|\Min(R)| = i_1i_2 \cdots i_m$.  If $\mathfrak q =  (X_{1k_1}, X_{2k_2}, \dots, X_{mk_m})R$ is another minimal prime of $R$, then $\mathfrak p$ and $\mathfrak q$ are adjacent in $\Gamma_R$ if and only if there exists an
integer $d$ between 1 and $m$ such that 
$|j_t - k_t| =  0$ if $t \neq d$, and $|j_d - k_d| \neq 0$.
Thus, the degree of $\mathfrak p$ in $\Gamma_R$ is $(i_1-1) + (i_2 -1) + \cdots + (i_m - 1)= i_1 + i_2 + \cdots + i_m - m.$  Hence, $\Gamma_R$ is connected and regular.  
\end{proof}


\section{Connected Components, Maximal Ideals, and Indecomposable Summands}
\label{sec140103c}


\begin{chunk} \label{blanket assumptions}
Throughout this section, we assume that $R$, in addition to being local, is complete and equidimensional with $n:=\dim (R)$.
Let $T$ be the $S_2$-ification of $R$ and $\omega$ the canonical module of $R$.  
\end{chunk}

This section is devoted to the proof of Theorem~\ref{thm140103a}, mostly presented in the following propositions, and completed in~\ref{proof140106a}.
Recall that the symbols $\summands$ and $\components$ are from Notation~\ref{notn140107a}.

\begin{proposition} \label{prop140106a}
Under the assumptions in \ref{blanket assumptions}, 
we have $\summands_R(\H^n_{\mathfrak m}(R))=\summands_R(\omega)$.
\end{proposition}
 
\begin{proof}  
Let $E$ be the injective hull of the residue field of $R$, and let $(-)^{\vee} = \Hom_R(-, E)$. 
By definition, we have $\H_{\mathfrak m}^n(R)\cong\omega^{\vee}$, so Matlis duality implies that
$\H_{\mathfrak m}^n(R)^{\vee}\cong\omega$. If $M$ is an indecomposable $R$-module (in particular, $M\neq 0$) that is  finitely generated (or artinian),
then Matlis duality implies that $M^\vee$ is an indecomposable $R$-module that is artinian  (or finitely generated). It follows that the  
decompositions of $\H_{\mathfrak m}^n(R)$ and $\omega$ into direct
sums of  indecomposables are in bijection, and the desired equality follows.
\end{proof}

\begin{proposition} \label{ciffb1} Under the assumptions in \ref{blanket assumptions}, 
we have $|m \sub \Spec(T)|=\summands_R(\omega)$.
\end{proposition}
 
\begin{proof}  
Case 1: $R$ is unmixed.
Set $s:=|m \sub \Spec(T)|$ and $t:=\summands_R(\omega)$.
Consider a decomposition $\omega = M_1 \oplus \cdots \oplus M_t$ such that each $M_i$ is indecomposable over $R$.
This gives rise to a system $e_1,\ldots,e_t$ of pairwise orthogonal idempotents in $\Hom_R(\omega,\omega)\cong T$,
namely, $e_i$ is the composition $M\to M_i\to M$ of the canonical maps given by the direct sum decomposition.
It follows that $T$ decomposes as a product $T\cong T_1\times\cdots\times T_t$ of non-zero rings.
In particular, we have 
$$s=|m \sub \Spec(T)|=|m \sub \Spec(T_1)|+\cdots+|m \sub \Spec(T_t)|\geq t.$$

For the reverse inequality, write $T = T_1 \times \cdots \times T_s$, where each $T_i$ is complete, local, and $(S_2)$.
Then the canonical module $\omega$ can be written as $\omega_1 \times \cdots \times \omega_s$, where each $\omega_i$
is the canonical module of $T_i$.  (See, e.g., \cite[(2.2) k)]{HH}, which requires that $R$ be unmixed.)
Since each $\omega_i$ is non-zero, we conclude that $s\leq \summands_R(\omega)=t$. This completes the proof in Case 1.

Case 2: the general case.
By~\cite[(2.2) d)]{HH}, we know that $\omega$ is a module over the complete, equidimensional, unmixed local ring
$\ol R:=R/j(R)$. (Moreover, $\omega$ is a canonical module for $\ol R$.)
It follows that the direct sum decompositions of $\omega$ into indecomposables over $R$ are in bijection with the direct sum decompositions of $\omega$ 
into indecomposables over $\ol R$, so we have $\summands_R(\omega)=\summands_{\ol R}(\omega)$.
By definition, $T$ is the $S_2$-ification of $\ol R$.
Thus, we have $|m \sub \Spec(T)|=\summands_{\ol R}(\omega)=\summands_R(\omega)$ by Case 1.
\end{proof}

\begin{proposition} \label{c=>d} Under the assumptions in \ref{blanket assumptions}, if  $J$ is an ideal of $R$ such that $\Ht_R(J) \geq 2$, then 
$\components(\Spec(R) -V(J))\leq |m \sub \Spec(T)|$.
\end{proposition}
 
\begin{proof}  
Set $s:=|m \sub \Spec(T)|$.

Since $s=1$ if and only if $\Spec(R) -V(J)$ is connected for every ideal $J$ of $R$ such that $\Ht_R(J) \geq 2$ by Fact \ref{HH}, assume that $s\geq 2$. Let $Q$ be a prime ideal of $T$ such that $Q \supseteq JT$.  Then $\Ht_T(Q) \geq \Ht_T(JT) = \Ht_R(J) \geq 2$, by \cite[(3.5) b)]{HH}.  
Since $T$ decomposes as a product of local rings
$T = T_1 \times \cdots \times T_s$ by \cite[(2.2) k)]{HH}, 
there exist unique $i$ and $Q_i \in \Spec(T_i)$ such that 
$Q=T_1 \times \cdots \times T_{i-1} \times Q_i \times T_{i+1} \times \cdots \times T_s$.  In other words, there is a containment-respecting bijection
$\Spec(T)\rightleftarrows \disjointunion_{i=1}^s\Spec(T_i)$.
It is straightforward to show that, under this bijection, 
we have $Q \in V(JT)$ if and only if $Q_i \in V(JT_i)$, that is,
we have another containment-respecting bijection
$\Spec(T)- V(JT) \rightleftarrows \disjointunion_{i=1}^s (\Spec(T_i)- V(JT_i))$. 
It follows that these bijections are homeomorphisms for the Zariski toplogies and subspace topologies.

Next, we claim that $\Ht_{T_i}(JT_i) \geq 2$.  It suffices to show that if $Q_i \in V(JT_i)$, then $\Ht_{T_i}(Q_i) \geq 2$.  Taking $Q$
as above, we know that $Q \supseteq JT$.  Hence $\Ht_{T_i}(Q_i) = \Ht_T(Q) \geq 2$, as desired.  Moreover, for each $i$, the set $\Spec(T_i)- V(JT_i)$ is non-empty since $\Ht_{T_i} (JT_i) \geq 2$.  (If $JT_i \subseteq \mathfrak p_i$ for all $\mathfrak p_i \in \Spec(T_i)$, then $\Ht_{T_i} (JT_i) =0$, a contradiction.)

The implication of this is that each $\Spec(T_i)- V(JT_i)$ is connected by Fact \ref{HH}.  To be specific, the ring $T_i$ satisfies the assumptions \ref{blanket assumptions} as well as the $(S_2)$ condition, so $T_i$ is its own $S_2$-ification; and since it is local, the equivalent conditions of Fact \ref{HH} apply. 
From this, we conclude that $\components(\Spec(T)-V(JT))=s$.

We claim that $\components(\Spec(R)- V(J))\leq s$. 
This will follow easily from an exercise in topology:  if $X \to Y$ is a continuous and surjective map of topological spaces, then $\components(X)\geq \components(Y)$.  In particular, we will apply it to a map $f\colon\Spec(T)- V(JT)\to\Spec(R)- V(J)$.  This map is induced from the map $F\colon\Spec(T) \to \Spec(R)$, which is given by contraction, and which is onto since $R \to T$ is an integral extension.  We need to show that $f$
is well-defined and surjective.  The map $f$ is well-defined because $F^{-1}(V(J)) = V(JT)$ (see, e.g., \cite[Exercise 1.21(ii)]{AM}).  
With the surjectivity of $F$, this also implies that $f$ is surjective, establishing the claim and the result.
\end{proof}

\begin{proposition} \label{prop140106b}
Under the assumptions in \ref{blanket assumptions}, 
we have $\components(\Gamma_R)=|m \sub \Spec(T)|$. 
\end{proposition}
 
\begin{proof}  Since $T$ is local if and only if $\Gamma_R$ is connected by Fact \ref{HH}, assume that $s:=|m \sub \Spec(T)|  \geq 2$.
Then $\dim (R) \geq 2$, as per Example \ref{dim1}.  
Write $T = T_1 \times \cdots \times T_s$, where each $T_i$ is complete, local, and $(S_2)$.
Let $\mathfrak N_1, \dots, \mathfrak N_s$ be the maximal ideals of $T$.  
Under the bijection $\Spec(T)\rightleftarrows \disjointunion_{i=1}^s\Spec(T_i)$
from the proof of Proposition~\ref{c=>d}, 
we have $\Min(T)\rightleftarrows \disjointunion_{i=1}^s\Min(T_i)$.
Next, by \cite[(3.5)]{HH}, there is a bijection between $\Min(T)\rightleftarrows\Min(R)$.
Let $\mathcal A_i\subseteq\Min(R)$ be the subset corresponding to $\Min(T_i)$.
As per the proof of \cite[(3.6)]{HH}, it is impossible to have an edge joining a vertex in $\mathcal A_i$ to a vertex in $\mathcal A_j$ when $i \neq j$.  
Therefore, the number of connected components of $\Gamma_R$ is greater than or equal to $s$, i.e., $\components(\Gamma_R)\geq|m \sub \Spec(T)|$.  

It remains to show that each subgraph of $\Gamma_R$ induced by $\mathcal A_i$
is connected.  We claim that this follows from the fact that each $T_i$ is complete, local, and $(S_2)$. 
Indeed, by Fact \ref{HH}, each graph $\Gamma_{T_i}$ is connected.  This means that any pair of distinct minimal primes in $T_i$ generates a height one ideal in $T_i$.  Thus, since the minimal primes of $T$ are in bijection with the minimal primes of $R$ and the height one primes of $T$ are in bijection with the height one primes of $R$ (as per \cite[(3.5)]{HH}), each subgraph composed of vertices from $\mathcal A_i$ is connected, as desired.
\end{proof}

In the next result, note that the case $\dim(R)\leq 1$ is treated in Example~\ref{dim1}.

\begin{proposition} \label{ht2} Under the assumptions in \ref{blanket assumptions}, 
if $\dim(R)\geq 2$, 
then there is an ideal $I$ of $R$ such that $\Ht_R (I) = 2$ and
$\components(\Spec(R) -V(I))= \components(\Gamma_R)$.
\end{proposition} 
 
\begin{proof}  
Set $t=\components(\Gamma_R)$.
If $t=1$, then the desired conclusion follows from Fact~\ref{HH} for every ideal of height 2. So we assume that $t\geq 2$ for the rest of this proof.

Write $\Gamma_R = \Gamma_1 \disjointunion \cdots \disjointunion \Gamma_t$, where the $\Gamma_i$ are the connected components of the graph $\Gamma_R$.  For  $i=1, \ldots t$, let $V_i=\{\mathfrak p_{i1}, \mathfrak p_{i2}, \dots, \mathfrak p_{i a_i} \}$ be the vertex set of $\Gamma_i$, that is, the set of
minimal primes of $R$ in the component $\Gamma_i$.   In particular, we are assuming that $a_i=|V_i|$ and $V_i\bigcap V_j=\emptyset$ for $i \neq j$.  For distinct elements $i,j\in\{1,\ldots, t\}$, set 
\begin{align*}
\fr_i &= \bigcap_{k = 1}^{a_i}  \fp_{ik}
&\fs_{ij} &= \fr_i + \fr_j
&I &= \bigcap_{i \neq j} \fs_{ij}.
\end{align*}

Define $V^{\circ}(\fr_i)$ to be those primes of $R$ containing $\fr_i$, but not $I$; i.e., $V^{\circ}(\fr_i)= \left(\Spec(R)- V(I)\right) \bigcap V(\fr_i)$.  

Claim 1: 
$\Spec(R)- V(I)=V^{\circ}(\fr_1) \disjointunion \cdots \disjointunion V^{\circ}(\fr_t)$.  
Since $\bigcap_i\fr_i$ is the intersection of all the minimal primes of $R$, we have 
$\Spec(R)- V(I)=V^{\circ}(\fr_1) \bigcup \cdots \bigcup V^{\circ}(\fr_t)$. 
If $P \in V^{\circ}(\fr_i) \bigcap V^{\circ}(\fr_j)$ for some $i \neq j$, then $P \supseteq \fr_i + \fr_j = \fs_{ij} \supseteq I$, a contradiction to the definition of $V^{\circ}(\fr_i)$.  Therefore, the union is disjoint.

Claim 2: $\Ht_R(I)\geq 2$.  Note that $\Ht_R(I) = \min\{\Ht_R(\fs_{ij}) \mid 1 \leq i < j \leq t \}$.  If $\Ht_R (\fs_{ij}) \leq 1$ for some pair $i \neq j$, then there exists a prime ideal $P$ of $R$ such that $\Ht_R(P)=1$ and
$$P \supseteq \fs_{ij}=\fr_i + \fr_j = (\mathfrak p_{i1} \bigcap \mathfrak p_{i2} \bigcap \dots \bigcap \mathfrak p_{i a_i}) + (\mathfrak p_{j1} \bigcap \mathfrak p_{j2} \bigcap \dots \bigcap \mathfrak p_{j a_j}).$$
It follows that $P\supseteq \mathfrak p_{i1} \bigcap \dots \bigcap \mathfrak p_{i a_i}$, so the primeness of $P$ implies that $P \supseteq \fp_{ik}$ for some $k$.
Similarly, we have $P \supseteq \fp_{jl}$ for some $l$,
so $P \supseteq \fp_{ik}+\fp_{jl}$.  It follows that $\Ht_R(\fp_{ik}+\fp_{jl})\leq 1$, so the vertices $\fp_{ik}$ and $\fp_{jl}$ of $\Gamma_R$ are adjacent.
This contradicts the fact that $V_i$ and $V_j$ are disjoint.  Therefore, we have $\Ht_R (\fs_{ij})  \geq 2$ for all $1 \leq i < j \leq t$; i.e., $\Ht_R (I) \geq 2$.  

Claim 3: each $V^{\circ}(\fr_i)$ is non-empty and closed in $\Spec(R)- V(I)$.
It is closed by definition of the topology of $\Spec(R)- V(I)$, which is induced by the Zariski topology on $\Spec(R)$.
Now, $V^{\circ}(\fr_i)$ is non-empty since $\fp_{im}\in V^{\circ}(\fr_i)$ for $m=1,\ldots,a_i$. Indeed, we have $\fp_{im}\in V(\fr_i)$ by definition of $\fr_i$.
And if $\fp_{im}\in V(I)$, then $\fp_{im}\supseteq\fs_{pq} = \fr_p + \fr_q$ for some $p\neq q$.
As in the proof of Claim 2, this implies that $\Ht_R(\fp_{im})\geq 1$, contradicting the minimality of $\fp_{im}$

If $\Ht_R (I) = 2$, then Claims 1--3 imply that
$I$ has the desired properties.  However, it is possible that $\Ht _R(I) \neq 2$.  We deal with this case now.  

Suppose that $\Ht_R (I) > 2$.  Then $\Ht_R (\fs_{ij}) \geq 3$ for each pair $i \neq j$.  For the first such pair (only), write
$\sqrt{\fs_{12}} = \bigcap_{\ell = 1}^m Q_{\ell}$ where each $Q_{\ell}$ is a prime such that $\Ht_R(Q_{\ell})\geq 3$.
Let $\mathfrak q$ be a height two prime in $Q_1$ that contains a minimal prime, say $\fp_{1j}$.  
Then $\mathfrak r_1 \subseteq \mathfrak q$.  Consider the ideals 
$$\mathfrak t_{12} = \mathfrak q \bigcap \left(\bigcap_{\ell \geq 2}  Q_{\ell}\right) \quad {\text{ and  } } \quad I^* = \mathfrak t_{12} \bigcap \left(\bigcap_{\{i,j\} \neq \{1, 2\}} \mathfrak s_{ij}\right).$$

Similarly as above, define $V^{*}(\fr_i)$ to be those primes of $R$ containing $\fr_i$, but not $I^*$; i.e., $V^{*}(\fr_i)= \left(\Spec(R)- V(I^*)\right) \bigcap V(\fr_i)$.  
If $P \in V^{*}(\fr_i) \bigcap V^{*}(\fr_j)$ for any $i \neq j$, then $P \supseteq \fr_i + \fr_j = \fs_{ij}$.  In the case that $\{i, j\} \neq \{1,2\}$, 
we have $\fs_{ij} \supseteq I^*$, a contradiction as above.  Otherwise, $P \supseteq \mathfrak r_1 + \mathfrak r_2= \mathfrak s_{12}$, and because $P$ is prime, $P \supseteq \sqrt{\mathfrak s_{12}} \supseteq \mathfrak t_{12} \supseteq I^*$, which is again a contradiction.  By construction, the height of $I^*$ is exactly two.  (To be specific, by construction $\mathfrak t_{12}$ has height two, and all of the other terms in the intersection defining $I^*$ have height at least two, by the work in the previous paragraph.)  In this case, $I^*$ has the desired properties.
\end{proof}

\begin{corollary} Under the assumptions in 4.1, if $\dim(R)=2$, then $\components(\Spec^{\circ}(R))=\beta(\Gamma_R)$ where $\Spec^{\circ}(R)=\Spec(R)-\{\m\}$ is the punctured spectrum of $R$.
\end{corollary}

\begin{proof} Since $\dim(R)=2$, the condition $\Ht_R(I)=2$ is equivalent to $\sqrt{I}=\m$.
\end{proof}

\begin{chunk}[Proof of Theorem~\ref{thm140103a}] \label{proof140106a}
If $\dim(R)\leq 1$ then the quantities (a)--(e) in Theorem~\ref{thm140103a} are all 1 by Example~\ref{dim1}.
If $\dim(R)\geq 2$, then the desired (in)equalities follow from Propositions~\ref{prop140106a}--\ref{ht2}.
\qed
\end{chunk}


\section{Graph Labeling and Realizing Graphs as $\Gamma_R$}
\label{sec140103e}


In this section, we investigate a labeling for graphs $G$ that allow us to construct rings $R$ 
such that $\Gamma_R$ is graph-isomorphic to $G$. Intuitively, the labeling works as follows. Each vertex of $G$ is assigned a distinct
``address'' consisting of $s$ distinct numbers, from a set of size $n$, such that two vertices are adjacent if and only if
their addresses differ by exactly one number. (Compare this with Fact~\ref{fact140103a}.) More precisely, we have the following.

\begin{notation}  Let $G$ be a graph with vertex set $V$, and set $d=|V|$.
Fix positive integers $n$ and $s$.
Set $[n]=\{1,\ldots,n\}$, and let
$\binom{[n]}{s}$ denote the set of subsets of $[n]$ with cardinality $s$.
\end{notation}

\begin{definition} \label{labels} 
An \emph{admissible labeling} of $G$ is an injective function $\phi\colon V\into \binom{[n]}{s}$,
for some choice of $n$ and $s$, satisfying
the following conditions:
\begin{enumerate}[(1)]
\item
$\phi(v_1)\bigcup\cdots\bigcup\phi(v_d)=[n]$, and
\item 
for all vertices $v$ and $w$, we have $v$ adjacent to $w$ in $G$ if and only if $|\phi(v)\bigcap\phi(w)|=s-1$,
that is, if and only if $|\phi(v)\bigcup\phi(w)|=s+1$.
\end{enumerate}
\end{definition}

\begin{remark}\label{rmk140103a}
Several notions of ``graph labelings'' exist in the literature. However, we have not been able to find this one in the literature.
\end{remark}

As the terminology suggests, we visualize admissible labelings by placing labels on the vertices of a graph, as in the following example.

\begin{example}  
Here are two admissible labelings of 
the Petersen graph with $d = 10$, $n = 6$, and $s= 3$.

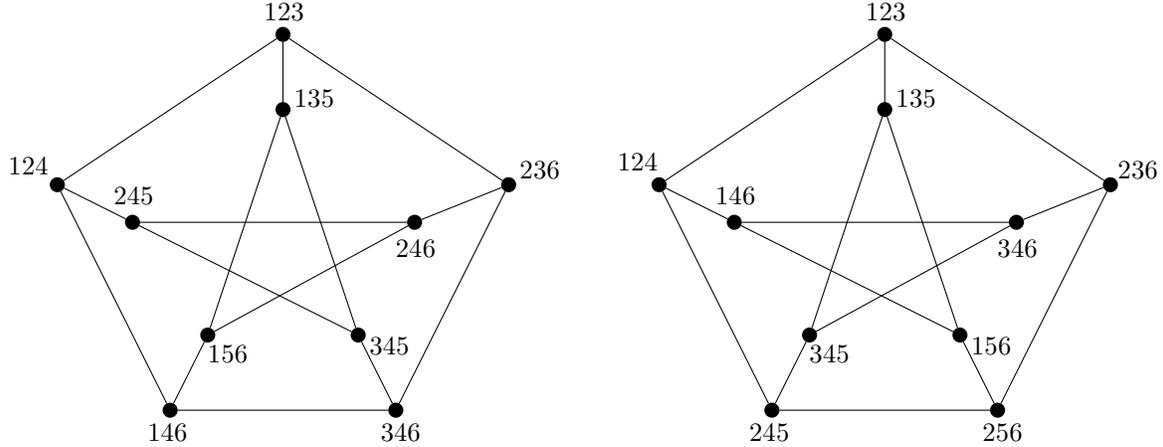
\begin{align*}
\begin{pgfpicture}{7cm}{-2cm}{-2cm}{4cm}
\pgfnodecircle{Node1}[fill]{\pgfxy(1.5,-1)}{0.1cm}
\pgfnodecircle{Node2}[fill]{\pgfxy(4.5,-1)}{0.1cm}
\pgfnodecircle{Node3}[fill]{\pgfxy(6, 2)}{0.1cm}
\pgfnodecircle{Node4}[fill]{\pgfxy(3, 4)}{0.1cm}
\pgfnodecircle{Node5}[fill]{\pgfxy(0,2)}{0.1cm}
\pgfnodecircle{Node6}[fill]{\pgfxy(2,0)}{0.1cm}
\pgfnodecircle{Node7}[fill]{\pgfxy(4,0)}{0.1cm}
\pgfnodecircle{Node8}[fill]{\pgfxy(4.75, 1.5)}{0.1cm}
\pgfnodecircle{Node9}[fill]{\pgfxy(3, 3)}{0.1cm}
\pgfnodecircle{Node10}[fill]{\pgfxy(1,1.5)}{0.1cm}
\pgfnodecircle{Node11}[fill]{\pgfxy(9.5,-1)}{0.1cm}
\pgfnodecircle{Node12}[fill]{\pgfxy(12.5,-1)}{0.1cm}
\pgfnodecircle{Node13}[fill]{\pgfxy(14, 2)}{0.1cm}
\pgfnodecircle{Node14}[fill]{\pgfxy(11, 4)}{0.1cm}
\pgfnodecircle{Node15}[fill]{\pgfxy(8,2)}{0.1cm}
\pgfnodecircle{Node16}[fill]{\pgfxy(10,0)}{0.1cm}
\pgfnodecircle{Node17}[fill]{\pgfxy(12,0)}{0.1cm}
\pgfnodecircle{Node18}[fill]{\pgfxy(12.75, 1.5)}{0.1cm}
\pgfnodecircle{Node19}[fill]{\pgfxy(11, 3)}{0.1cm}
\pgfnodecircle{Node20}[fill]{\pgfxy(9,1.5)}{0.1cm}
\pgfnodeconnline{Node1}{Node2}
\pgfnodeconnline{Node2}{Node3}
\pgfnodeconnline{Node3}{Node4}
\pgfnodeconnline{Node4}{Node5}
\pgfnodeconnline{Node1}{Node5}
\pgfnodeconnline{Node1}{Node6}
\pgfnodeconnline{Node2}{Node7}
\pgfnodeconnline{Node3}{Node8}
\pgfnodeconnline{Node4}{Node9}
\pgfnodeconnline{Node5}{Node10}
\pgfnodeconnline{Node6}{Node8}
\pgfnodeconnline{Node6}{Node9}
\pgfnodeconnline{Node7}{Node9}
\pgfnodeconnline{Node7}{Node10}
\pgfnodeconnline{Node8}{Node10}
\pgfnodeconnline{Node11}{Node12}
\pgfnodeconnline{Node12}{Node13}
\pgfnodeconnline{Node13}{Node14}
\pgfnodeconnline{Node14}{Node15}
\pgfnodeconnline{Node11}{Node15}
\pgfnodeconnline{Node11}{Node16}
\pgfnodeconnline{Node12}{Node17}
\pgfnodeconnline{Node13}{Node18}
\pgfnodeconnline{Node14}{Node19}
\pgfnodeconnline{Node15}{Node20}
\pgfnodeconnline{Node16}{Node18}
\pgfnodeconnline{Node16}{Node19}
\pgfnodeconnline{Node17}{Node19}
\pgfnodeconnline{Node17}{Node20}
\pgfnodeconnline{Node18}{Node20}
\pgfputat{\pgfxy(1.2, -1.3)}{\pgfbox[left,center]{146}}
\pgfputat{\pgfxy(4.3,-1.3)}{\pgfbox[left,center]{346}}
\pgfputat{\pgfxy(6.15, 2.2)}{\pgfbox[left,center]{236}}
\pgfputat{\pgfxy(2.75, 4.3)}{\pgfbox[left,center]{123}}
\pgfputat{\pgfxy(-.65, 2.25)}{\pgfbox[left,center]{124}}
\pgfputat{\pgfxy(2, -.25)}{\pgfbox[left,center]{156}}
\pgfputat{\pgfxy(4.15,-.15)}{\pgfbox[left,center]{345}}
\pgfputat{\pgfxy(4.5, 1.15)}{\pgfbox[left,center]{246}}
\pgfputat{\pgfxy(3.15, 3.15)}{\pgfbox[left,center]{135}}
\pgfputat{\pgfxy(.75, 1.85)}{\pgfbox[left,center]{245}}
\pgfputat{\pgfxy(9.2, -1.3)}{\pgfbox[left,center]{245}}
\pgfputat{\pgfxy(12.3,-1.3)}{\pgfbox[left,center]{256}}
\pgfputat{\pgfxy(14.1, 2.2)}{\pgfbox[left,center]{236}}
\pgfputat{\pgfxy(10.75, 4.3)}{\pgfbox[left,center]{123}}
\pgfputat{\pgfxy(7.45, 2.3)}{\pgfbox[left,center]{124}}
\pgfputat{\pgfxy(10, -.25)}{\pgfbox[left,center]{345}}
\pgfputat{\pgfxy(12.15,-.15)}{\pgfbox[left,center]{156}}
\pgfputat{\pgfxy(12.5, 1.15)}{\pgfbox[left,center]{346}}
\pgfputat{\pgfxy(11.15, 3.15)}{\pgfbox[left,center]{135}}
\pgfputat{\pgfxy(8.75, 1.85)}{\pgfbox[left,center]{146}}
\end{pgfpicture}
\end{align*}
It is straightforward (though tedious) to show that these labelings are distinct up to graph isomorphism
and permutation of the elements of $[6]$, and that these are the only two admissible labelings with $d = 10$, $n = 6$, and $s= 3$, up to graph isomorphism
and permutation of the elements of $[6]$.
\end{example}


\begin{lemma} \label{lem140103b} Let $G$ be a graph.
\begin{enumerate}
\item If $G$ has an admissible labeling $\phi\colon V\into \binom{[n]}{s}$, then so does each induced subgraph $G'$. 
\item $G$ has an admissible labeling $\phi\colon V\into \binom{[n]}{1}$ if and only if $G$ is complete. 
\end{enumerate}
\end{lemma}

\begin{proof} 
(1) Let $V'$ be the vertex set for $G'$, and re-order the elements of $[n]$ to assume that
$\bigcup_{v\in V'}\phi(v)$ is of the form $[n']$ for some $n'\leq n$. 
Define $\phi'\colon V'\into \binom{[n']}{s}$ by the formula $\phi'(v):=\phi(v)$. 
Since two vertices in $V'$ are adjacent in $G'$ if and only if they are adjacent in $G$, it follows readily by definition
that $\phi'$ is an admissible labeling of $G'$.

(2) The proof of this is straightforward.
\end{proof}

\begin{remark} 
The converse of Lemma \ref{lem140103b}(1) also holds trivially since $G$ is an induced subgraph of itself. 
However, there exist graphs $G$ such that every \emph{proper} induced subgraph has an admissible labeling, but $G$ does not admit an admissible labeling. Specifically, Proposition \ref{prop140103a} exhibits a graph on five vertices that does not have an admissible labeling. 
Note that this example has the smallest possible number of vertices, since every graph on at most four vertices has an admissible labeling.
Indeed, $C_4, P_3, K_4 , K_{1,3}$, and the totally disconnected graph fall under the purview of Proposition \ref{prop140103ax}.  The graph $K_{1,1,2}$ appears in the proof of Proposition \ref{prop140103a}, and the remaining connected graph, which is a triangle with one pendant, can easily be obtained from Graph (8) in Remark \ref{rmk140107a} by deleting the extra edge.  The other graphs are disconnected and can obtained from components of smaller graphs, all of which are detailed in Examples~\ref{dim1}-\ref{min=3}.

\end{remark}

\begin{para}[Proof of Theorem~\ref{thm140103b}]\label{proof140103a}
Let $\phi\colon V\into \binom{[n]}{s}$ be an admissible labeling of $G$.
Set $Q=\mathsf k[\![X_1,\ldots,X_n]\!]$.
For each subset $\mathcal A=\{i_1,\ldots,i_s\}$ in $\binom{[n]}{s}$, set $P_{\mathcal A}=(X_{i_1},\ldots,X_{i_s})Q$.
Define $I=\bigcap_{v\in V}P_{\phi(v)}$, and set $R=Q/I$. 
It follows that $I$ is generated by square-free monomials in the variables $X_1,\ldots,X_n$, and 
the minimal primes of $R$ are exactly the ideals of the form $P_{\phi(v)}R$ with $v\in V$.
The fact that $R$ is equidimensional such that
$\Gamma_R$ is isomorphic to $G$ follows from a direct comparison of Fact~\ref{fact140103a} and Definition~\ref{labels}.
\qed
\end{para}

\begin{remark}  \label{rmk140104b}
The rings in Examples~\ref{ex140103b}, \ref{min=3}, and~\ref{cycle} are constructed as in the preceding proof.
In the first three of these examples, one can see the admissible labelings by inspecting the vertex ideals.
For instance, in Example~\ref{ex140103b} we have the following admissible labelings.

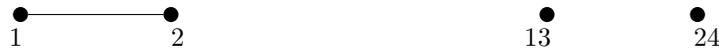
\begin{align*}
\begin{pgfpicture}{3cm}{-.5cm}{6cm}{0cm}
\pgfnodecircle{Node1}[fill]{\pgfxy(0,0)}{0.1cm}
\pgfnodecircle{Node2}[fill]{\pgfxy(2,0)}{0.1cm}
\pgfnodecircle{Node3}[fill]{\pgfxy(7, 0)}{0.1cm}
\pgfnodecircle{Node4}[fill]{\pgfxy(9, 0)}{0.1cm}
\pgfnodeconnline{Node1}{Node2}
\pgfputat{\pgfxy(-.15, -.3)}{\pgfbox[left,center]{$1$}}
\pgfputat{\pgfxy(2,-.3)}{\pgfbox[left,center]{$2$}}
\pgfputat{\pgfxy(6.7,-.3)}{\pgfbox[left,center]{$13$}}
\pgfputat{\pgfxy(8.95, -.3)}{\pgfbox[left,center]{$24$}}
\end{pgfpicture}
\end{align*}
Considering the first graph, our construction yields the ring $R=\mathsf k[\![X_1,X_2]\!]/I$ where
$I=(X_1)\bigcap(X_2)=(X_1X_2)$.
For the second graph, our construction yields the ring $R=\mathsf k[\![X_1,X_2,X_3,X_4]\!]/I$ where
$I=(X_1,X_3)\bigcap(X_2,X_4)=(X_1X_2,X_2X_3,X_3X_4,X_1X_4)$.
\end{remark} 

In preparation for the proof, the next few results give bounds on the numbers $n$, $d$, and $s$ from Definition~\ref{labels}. 

\begin{lemma}  \label{upperbound} 
Let $G$ be a  graph with an admissible labeling $\phi\colon V\into \binom{[n]}{s}$.
Let $v_1, \dots, v_m$ be vertices in $G$ such that the subgraph of $G$ induced by $v_1, \dots, v_m$ is connected.  
Then $\left|\phi(v_1) \bigcup \cdots \bigcup \phi(v_m)\right| \leq s + m - 1$.  
\end{lemma}

\begin{proof} 
We argue by induction on $m$.  If $m = 1$, then $|\phi(v_1)| = s= s+1-1$, and the base case is established.  
Now assume the claim is true for lists of  $m$ vertices, and consider vertices $v_1, \dots, v_m, v_{m+1}$
such that the induced  subgraph $G'$ of $G$ is connected. 
Let $T$ be a spanning tree of $G'$. 
Since $T$ is a tree, we may re-order the vertices if necessary to assume  that the subgraph of $T$ induced by $v_1, \dots, v_m$ is also connected
and $v_{m+1}$ is adjacent to $v_1$ in $T$.
(For instance, let $v_{m+1}$ be a pendant vertex, i.e., a leaf, that is adjacent to $v_1$ in $T$.)
The inclusion-exclusion principle implies that
$$\left|\phi(v_1) \bigcup \cdots \bigcup \phi(v_{m+1})\right| = \left|\phi(v_1) \bigcup \cdots \bigcup \phi(v_m)\right| + |\phi(v_{m+1})| - \left|\left(\phi(v_1) \bigcup \cdots \bigcup \phi(v_m)\right) \bigcap \phi(v_{m+1})\right|.$$
By the induction hypothesis, the first term on the right-hand side of this equation is less than or equal to $s+m-1$.  Consider the third term: 
$$\left|\left(\phi(v_1) \bigcup \cdots \bigcup \phi(v_m)\right) \bigcap \phi(v_{m+1})\right| = \left|\left(\phi(v_1) \bigcap \phi(v_{m+1})\right) \bigcup \cdots \bigcup \left(\phi(v_m) \bigcap \phi(v_{m+1})\right)\right|.$$
Since $v_{m+1}$ is adjacent to $v_1$, we have $\left|\left(\phi(v_1) \bigcap \phi(v_{m+1})\right)\right|=s-1$,
and it follows that 
$$ \left|\left(\phi(v_1) \bigcap \phi(v_{m+1})\right) \bigcup \cdots \bigcup \left(\phi(v_m) \bigcap \phi(v_{m+1})\right)\right| \geq s-1.$$  
Therefore, we have
$$\left|\phi(v_1) \bigcup \cdots \bigcup \phi(v_{m+1})\right| \leq (s+m-1) + s - (s-1) = s + (m+1) - 1$$
as desired.  
\end{proof}

\begin{proposition}  \label{pupperbound} 
If $G$ is a connected graph with an admissible labeling $\phi\colon V\into \binom{[n]}{s}$,
then $n \leq s+d-1$.
\end{proposition}

\begin{proof}
The vertex set $\{v_1,\ldots,v_d\}$ of $G$ satisfies the hypotheses of Lemma~\ref{upperbound}, with $\phi(v_1)\bigcup\cdots\bigcup\phi(v_d)=[n]$ by definition of admissible labeling..
\end{proof}

\begin{lemma} \label{lowerbound} 
Let $G$ be a  graph with an admissible labeling $\phi\colon V\into \binom{[n]}{s}$.
Let $v_1, \dots, v_m$ be vertices in $G$ such that the subgraph of $G$ induced by $v_1, \dots, v_m$ is connected.  
Then $\left|\phi(v_1) \bigcap \cdots \bigcap \phi(v_m)\right| \geq s - m + 1$.  
\end{lemma}

\begin{proof} 
We argue by induction on $m$.  If $m = 1$, then $|\phi(v_1)| = s= s-1+1$, and the base case is established.  
Now assume the claim is true for lists of  $m$ vertices, and consider vertices $v_1, \dots, v_m, v_{m+1}$
such that the induced  subgraph $G'$ of $G$ is connected. 
Let $T$ be a spanning tree of $G'$. 
Since $T$ is a tree, we may re-order the vertices if necessary to assume  that the subgraph of $T$ induced by $v_1, \dots, v_m$ is also connected
and $v_{m+1}$ is adjacent to $v_1$ in $T$.
The inclusion-exclusion principle yields
\begin{multline} \label{eq140107a} 
\left|\left(\phi(v_1) \bigcap \cdots \bigcap \phi(v_m)\right) \bigcup \phi(v_{m+1})\right| \\
=  \left|\phi(v_1) \bigcap \cdots \bigcap \phi(v_m)\right| + |\phi(v_{m+1})| - \left|\left(\phi(v_1) \bigcap \cdots \bigcap \phi(v_m)\right) \bigcap \phi(v_{m+1})\right|.
\end{multline}
By the induction hypothesis, the first term of the right-hand side of this equation is greater than or equal to $s-m+1$.  Consider the left-hand side, rewritten as:
$$\left|\left(\phi(v_1) \bigcup \phi(v_{m+1})\right) \bigcap \cdots \bigcap \left(\phi(v_m) \bigcup \phi(v_{m+1})\right)\right|\leq \left|\phi(v_1) \bigcup \phi(v_{m+1})\right|=s+1.$$
Therefore,  equation~\eqref{eq140107a} implies that
$$s+1 \geq \left|\left(\phi(v_1) \bigcap \cdots \bigcap \phi(v_m)\right) \bigcup \phi(v_{m+1})\right| \geq (s-m+1) + s - \left|\phi(v_1) \bigcap \cdots \bigcap \phi(v_m) \bigcap \phi(v_{m+1})\right|$$
from which it follows that 
$\displaystyle{\left|\phi(v_1) \bigcap \cdots \bigcap \phi(v_m) \bigcap \phi(v_{m+1})\right| \geq s - (m+1) + 1}.$
\end{proof}

\begin{proposition}  \label{plowerbound} 
If $G$ is a connected graph with an admissible labeling $\phi\colon V\into \binom{[n]}{s}$,
then  $$\left|\phi(v_1) \bigcap \cdots \bigcap \phi(v_d)\right| \geq s-d+1.$$
\end{proposition}

\begin{proof}
The vertex set $\{v_1,\ldots,v_d\}$ of $G$ satisfies the hypotheses of Lemma~\ref{lowerbound}.
\end{proof}

\begin{corollary}\label{cor140103a}
Let $G$ be a connected graph with an admissible labeling $\phi\colon V\into \binom{[n]}{s}$.
If $d\geq 2$, then $G$ has a second admissible labeling $\phi'\colon V\into \binom{[n-s+d-1]}{d-1}$.
\end{corollary}

\begin{proof}
If $s=d-1$, then there is nothing to prove, so assume that $s\neq d-1$.  In the case where $s\geq d\geq 2$, Proposition~\ref{plowerbound} implies that $$\left|\phi(v_1) \bigcap \cdots \bigcap \phi(v_d)\right|\geq s-d+1\geq 1.$$
Re-order the set $[n]$ if necessary to assume that we have $n-s+d,\ldots,n\in\phi(v_1) \bigcap \cdots \bigcap \phi(v_d)$,
and define $\phi'\colon V\into \binom{[n-s+d-1]}{d-1}$ as $\phi'(v):=\phi(v)-\{n-s+d,\ldots,n\}$. Since $\phi$ is an admissible labeling of $G$, it is straightforward
to show that $\phi'$ is also. 

In the case $s < d-1$,  define $\phi'\colon V\into \binom{[n-s+d-1]}{d-1}$ as $\phi'(v):=\phi(v)\bigcup\{n+1,\ldots,n-s+d\}$. 
\end{proof}

Our next result shows that the bounds from Propositions \ref{pupperbound} and \ref{plowerbound} are sharp; see graph (9) in Remark \ref{rmk140107a} below. We repeatedly make use of the fact that one can re-order (i.e., permute) the elements of $[n]$ using an element of the symmetric group $S_n$ to a given admissible labeling. This allows us to put some labels into specific forms (e.g., $\phi(v_d)=\{1,2,\ldots,s\}$) to make for easier bookkeeping.


\begin{proposition}  \label{sharpbound} 
Let $G$ be the star graph on $d \geq 2$ vertices, i.e., the complete bipartite graph $K_{1,d-1}$. 
Then $G$ has an admissible labeling $\phi\colon V\to\binom{[2(d-1)]}{d-1}$
such that $\phi(v_1)\bigcap\cdots\bigcap\phi(v_d)=\emptyset$. Furthermore, any admissible labeling $\psi\colon V\into\binom{[n]}{s}$ of $G$ 
has $s\geq d-1$ and $n= s+d-1 \geq 2(d-1)$ and $\left|\phi(v_1)\bigcap\cdots\bigcap\phi(v_d)\right|=s-d+1$.
\end{proposition}

\begin{proof}  
By definition, $G$ has a vertex $v_d$ with degree $d-1$ and all other vertices $v_1,\ldots,v_{d-1}$ have degree 1.
(Note that $v_d$ is uniquely determined unless $d=2$.)

Define $\phi\colon V\to\binom{[2(d-1)]}{d-1}$ as follows:
$\phi(v_d)=\{1,\ldots,d-1\}$ and $\phi(v_i)=\{1,\ldots,d-1\}-\{i\}\bigcup\{d-1+i\}$ for $i=1,\ldots,d-1$.
For example, in the case $d > 4$, we have $\phi(v_1)=\{2,3,\ldots,d\}$,
$\phi(v_2)=\{1,3,\ldots,d-1,d+1\}$,
and $\phi(v_3)=\{1,2,4,\ldots,d-1,d+2\}$.
It is straightforward to verify that for $i<j<d$ we have
$\phi(v_i)\bigcap\phi(v_d)=\{1,\ldots,d-1\}-\{i\}$
and
$\phi(v_i)\bigcap\phi(v_j)=\{1,\ldots,d-1\}-\{i,j\}$.
Moreover, we have $\phi(v_1)\bigcup\cdots\bigcup\phi(v_d)=\{1,\ldots,2(d-1)\}=[2(d-1)]$,
so $\phi$ is an admissible labeling of $G$.  From the explicit description of $\phi$, it is straightforward to show that $\phi(v_1)\bigcap\cdots\bigcap\phi(v_d)=\emptyset$.

Now, suppose that $\psi\colon V\into\binom{[n]}{s}$ is an admissible labeling. 

Claim: The elements of $[n]$ can be re-ordered so that we have $\psi(v_i)=(\psi(v_d)-\{i\})\bigcup\{s+i\}$ for $i=1,\ldots,d-1$.  To prove this, start by re-ordering the elements of $[n]$ to assume that $\psi(v_d)=\{1,\ldots,s\}$.
Consider the edge $v_1v_d$. Since $\left|\psi(v_1)\bigcap\psi(v_d)\right|=s-1$, we have
$\psi(v_1)=(\psi(v_d)-\{a\})\bigcup\{b\}$ for some $a\in[s]$ and some $b\in[n]-[s]$.
Thus, we can re-order the elements of $[n]$
to assume that $\psi(v_1)=(\psi(v_d)-\{1\})\bigcup\{s+1\}=\{2,\ldots,s,s+1\}$.
Next, consider the edge $v_2v_d$. 
As with the previous edge, we have $\psi(v_2)=(\psi(v_d)-\{p\})\bigcup\{q\}$ for some $p\in[s]$ and some $q\in[n]-[s]$.
If $p=1$, then we have $2,\ldots,s\in\psi(v_1)\bigcap\psi(v_2)$; however, $v_1$ is not adjacent to $v_2$, so we must have
$\left|\psi(v_1)\bigcap\psi(v_2)\right|\leq s-2$, a contradiction. 
It follows that we must have $2\leq p\leq s$ so we can re-order  the set $\{2,\ldots,s\}$ to assume that $p=2$.
Similarly, we must have $q>s+1$, so we can re-order the set $\{s+2,\ldots,n\}$ to assume that $q=s+2$.
Continue in this way for the edges $v_iv_d$ with $i=3,\ldots,d-1$ to complete the proof of the claim.

From the claim, we must have $1,\ldots,d-1\in [s]$. It follows that $s\geq d-1$, establishing the first conclusion of our result.
For the second conclusion, note that the sets $\psi(v_d),\psi(v_1),\psi(v_2),\ldots,\psi(v_{d-1})$ are
$\{1,\ldots,s\}$, $\{2,\ldots,s,s+1\}$, $\{1,3,\ldots,s,s+2\}$, \ldots, $\{1,2,\ldots,d-2,d,\ldots,s,s+d-1\}$.
From this description, we see that the largest integer occurring in any set $\psi(v_p)$ is $s+d-1$.
Since $\bigcup_p\psi(v_p)=[n]$, it follows that the largest number $n$ in this set is $s+d-1$.
For the final conclusion, use the preceding description to observe that
$\phi(v_1)\bigcap\cdots\bigcap\phi(v_d)=\{d,\ldots,s\}$, which has cardinality $s-d+1$, as desired.
\end{proof}

Next, we present a graph without an admissible labeling; see also Remark~\ref{rmk140107a}.

\begin{proposition} \label{prop140103a}
The graph $G$ below does not have an admissible labeling.

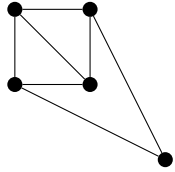
\begin{align*}
\begin{pgfpicture}{0cm}{-1.5cm}{2cm}{1cm}
\pgfnodecircle{Node1}[fill]{\pgfxy(0,0)}{0.1cm}
\pgfnodecircle{Node2}[fill]{\pgfxy(1,0)}{0.1cm}
\pgfnodecircle{Node3}[fill]{\pgfxy(1, 1)}{0.1cm}
\pgfnodecircle{Node4}[fill]{\pgfxy(0, 1)}{0.1cm}
\pgfnodecircle{Node5}[fill]{\pgfxy(2,-1)}{0.1cm}
\pgfnodeconnline{Node1}{Node2}
\pgfnodeconnline{Node2}{Node4}
\pgfnodeconnline{Node1}{Node4}
\pgfnodeconnline{Node2}{Node3}
\pgfnodeconnline{Node3}{Node4}
\pgfnodeconnline{Node1}{Node5}
\pgfnodeconnline{Node3}{Node5}
\end{pgfpicture}
\end{align*}
\end{proposition}  

\begin{proof}
We name the vertices of $G$ as $A,\ldots,E$ as follows.  Note that these are not labels for the vertices (as from an admissible labeling).

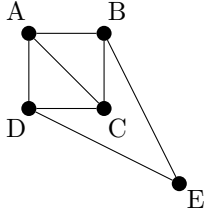
\begin{align*}
\begin{pgfpicture}{0cm}{-1.5cm}{2cm}{1cm}
\pgfnodecircle{Node1}[fill]{\pgfxy(0,0)}{0.1cm}
\pgfnodecircle{Node2}[fill]{\pgfxy(1,0)}{0.1cm}
\pgfnodecircle{Node3}[fill]{\pgfxy(1, 1)}{0.1cm}
\pgfnodecircle{Node4}[fill]{\pgfxy(0, 1)}{0.1cm}
\pgfnodecircle{Node5}[fill]{\pgfxy(2,-1)}{0.1cm}
\pgfnodeconnline{Node1}{Node2}
\pgfnodeconnline{Node2}{Node4}
\pgfnodeconnline{Node1}{Node4}
\pgfnodeconnline{Node2}{Node3}
\pgfnodeconnline{Node3}{Node4}
\pgfnodeconnline{Node1}{Node5}
\pgfnodeconnline{Node3}{Node5}
\pgfputat{\pgfxy(-.3, -.25)}{\pgfbox[left,center]{D}}
\pgfputat{\pgfxy(1.05,-.25)}{\pgfbox[left,center]{C}}
\pgfputat{\pgfxy(1.05, 1.3)}{\pgfbox[left,center]{B}}
\pgfputat{\pgfxy(-.3, 1.3)}{\pgfbox[left,center]{A}}
\pgfputat{\pgfxy(2.1, -1.2)}{\pgfbox[left,center]{E}}
\end{pgfpicture}
\end{align*}

Suppose by way of contradiction that the given graph $G$ has an admissible labeling $\phi\colon V\into \binom{[n]}{s}$.
Since we have $d=5$, Corollary~\ref{cor140103a} implies that we may assume that $s\leq 4$.
Since $G$ is not complete, Lemma~\ref{lem140103b} implies that $s\geq 2$. 
As in the proof of Corollary~\ref{cor140103a}, we may assume without loss of generality that
$\phi(A)\bigcap\cdots\bigcap\phi(E)=\emptyset$.

Case 1: $s=2$.  
Re-order the elements of $[n]$ to assume that $\phi(A)=\{1,2\}$. 
By definition we have $\left|\phi(A)\bigcap\phi(B)\right|=1$,
so we re-order the elements of $[n]$ to assume that $\phi(B)=\{1,3\}$. 
Since $\left|\phi(A)\bigcap\phi(C)\right|=1=\left|\phi(B)\bigcap\phi(C)\right|$, we 
re-order again to assume that either $\phi(C) = \{1, 4\}$ or $\phi(C) = \{2,3\}$, depending on whether $1\in\phi(C)$ or $1\notin\phi(C)$.

Sub-case 1a: $\phi(C) = \{1, 4\}$ 
In this case, we must have $D=\{2,4\}$. (Indeed, since $B$ and $D$ are not adjacent, we  have $1\notin\phi(D)$. Since $A$ and $D$ are adjacent,
we therefore must have $2\in\phi(D)$. And since $C$ and $D$ are adjacent,
we  must have $4\in\phi(D)$.)
Since $A$ and $C$ are not adjacent to $E$, we must have $1,2,3\notin\phi(E)$, but this implies that $\phi(D)\bigcap\phi(E)=\emptyset$, contradicting the fact that
$D$ and $E$ are adjacent.

Sub-case 1b: $\phi(C) = \{2,3\}$
In this sub-case, as in the previous one, we  have $D=\{2,4\}$ or $D=\{2,5\}$ after re-ordering, and a contradiction is arrived in a similar manner. 

Case 2:  $s=3$.  As in Case 1, re-order the elements of $[n]$ to assume that $\phi(A) =\{1, 2, 3\}$ and $\phi(B) = \{1, 2, 4\}$
and either $\phi(C) = \{1, 2, 5\}$,   $\phi(C) = \{1, 3, 4\}$, or $\phi(C) = \{2, 3, 4\}$. 
Suppose that $\phi(C) = \{1, 2, 5\}$. As above, re-order the elements of $[n]$ to assume that 
$\phi(D)=\{1, 3, 5\}$ or $\phi(D)=\{2, 3, 5\}$. 
Suppose that $\phi(D)=\{1, 3, 5\}$. It follows that
$\phi(A)\bigcap\cdots\bigcap\phi(D)=\{1\}$.
By assumption, we have $\phi(A)\bigcap\cdots\bigcap\phi(E)=\emptyset$, so we conclude that $1\notin\phi(E)$.
Using the edges $BE$ and $DE$, we conclude that $2,4,3,5\in\phi(E)$, contradicting the assumption $|\phi(E)|=s=3$.
The remaining sub-cases (as depicted below) are handled similarly.
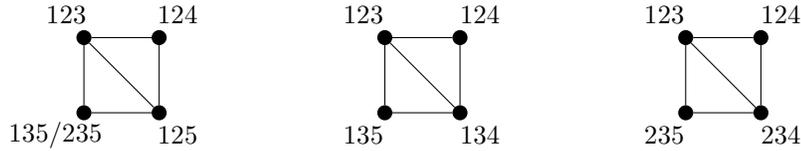
\begin{align*}
\begin{pgfpicture}{5cm}{-1cm}{1cm}{2cm}
\pgfnodecircle{Node1}[fill]{\pgfxy(0,0)}{0.1cm}
\pgfnodecircle{Node2}[fill]{\pgfxy(1,0)}{0.1cm}
\pgfnodecircle{Node3}[fill]{\pgfxy(1, 1)}{0.1cm}
\pgfnodecircle{Node4}[fill]{\pgfxy(0, 1)}{0.1cm}
\pgfnodecircle{Node5}[fill]{\pgfxy(8,0)}{0.1cm}
\pgfnodecircle{Node6}[fill]{\pgfxy(9,0)}{0.1cm}
\pgfnodecircle{Node7}[fill]{\pgfxy(9, 1)}{0.1cm}
\pgfnodecircle{Node8}[fill]{\pgfxy(8, 1)}{0.1cm}
\pgfnodecircle{Node9}[fill]{\pgfxy(4,0)}{0.1cm}
\pgfnodecircle{Node10}[fill]{\pgfxy(5,0)}{0.1cm}
\pgfnodecircle{Node11}[fill]{\pgfxy(5, 1)}{0.1cm}
\pgfnodecircle{Node12}[fill]{\pgfxy(4, 1)}{0.1cm}
\pgfnodeconnline{Node1}{Node2}
\pgfnodeconnline{Node2}{Node4}
\pgfnodeconnline{Node1}{Node4}
\pgfnodeconnline{Node2}{Node3}
\pgfnodeconnline{Node3}{Node4}
\pgfnodeconnline{Node5}{Node6}
\pgfnodeconnline{Node6}{Node8}
\pgfnodeconnline{Node5}{Node8}
\pgfnodeconnline{Node6}{Node7}
\pgfnodeconnline{Node7}{Node8}
\pgfnodeconnline{Node9}{Node12}
\pgfnodeconnline{Node10}{Node11}
\pgfnodeconnline{Node10}{Node12}
\pgfnodeconnline{Node9}{Node10}
\pgfnodeconnline{Node11}{Node12}
\pgfputat{\pgfxy(-1, -.3)}{\pgfbox[left,center]{$135/235$}}
\pgfputat{\pgfxy(.98,-.3)}{\pgfbox[left,center]{$125$}}
\pgfputat{\pgfxy(.98, 1.3)}{\pgfbox[left,center]{$124$}}
\pgfputat{\pgfxy(-.5, 1.3)}{\pgfbox[left,center]{$123$}}
\pgfputat{\pgfxy(7.45, -.3)}{\pgfbox[left,center]{$235$}}
\pgfputat{\pgfxy(9,-.3)}{\pgfbox[left,center]{$234$}}
\pgfputat{\pgfxy(9, 1.3)}{\pgfbox[left,center]{$124$}}
\pgfputat{\pgfxy(7.45, 1.3)}{\pgfbox[left,center]{$123$}}
\pgfputat{\pgfxy(3.45, -.3)}{\pgfbox[left,center]{$135$}}
\pgfputat{\pgfxy(5,-.3)}{\pgfbox[left,center]{$134$}}
\pgfputat{\pgfxy(5, 1.3)}{\pgfbox[left,center]{$124$}}
\pgfputat{\pgfxy(3.45, 1.3)}{\pgfbox[left,center]{$123$}}
\end{pgfpicture}
\end{align*}
Case 3: $s=4$.  
As above, we assume that $\phi(A) =\{1, 2, 3, 4\}$ and $\phi(B) = \{1, 2, 3, 5\}$.
After re-ordering, it follows that $\phi(C) = \{1, 2, 3, 6\}$ or $\phi(C) = \{n,m, 4, 5\}$, where $n, m \in \{1, 2, 3\}$.  
Suppose that $\phi(C) = \{1, 2, 3, 6\}$. Again after re-ordering, we must have $\phi(D)=\{1, 2, 4, 6\}$, $\phi(D)=\{1, 3, 4, 6\}$, or $\phi(D)=\{2, 3, 4, 6\}$.
Suppose that $\phi(D)=\{1, 2, 4, 6\}$. It follows that $\phi(A) \bigcap\cdots \bigcap \phi(D) = \{1,2\}$.  Since we 
have $\phi(A) \bigcap \cdots\bigcap \phi(E)=\emptyset$, we must have $1,2\notin\phi(E)$. But this implies that $\phi(B)\bigcap\phi(E)\subseteq\{4,6\}$,
hence $3=\left|\phi(B)\bigcap\phi(E)\right|\leq 2$, a contradiction. The remaining sub-cases are handled similarly.
\end{proof}

Here is a list of some classes of graphs that have admissible labelings.

\begin{proposition}  \label{prop140103ax}
The following graphs have admissible labelings:
\begin{enumerate}
\item Any path $P_d$;
\item Any cycle $C_d$;
\item Any complete graph $K_d$;
\item Any graph which is totally disconnected; i.e., just a set of discrete points;
\item Any star graph.
\end{enumerate}
\end{proposition}

\begin{proof} 
Items (2) and (3) follow from Examples \ref{cycle} and  \ref{ex140103a}, respectively, as in Remark~\ref{rmk140104b}.  
Item (5) is from Proposition~\ref{prop140103a}.
An admissible labeling for the path $v_1-v_2-\cdots-v_d$ is 
$\phi(v_i)=\{i,i+1\}$.
And the disjoint union of vertices $v_1,\ldots,v_d$ has admissible labeling $\phi(v_i)=\{2i-1,2i\}$.
\end{proof}

\begin{remark}  \label{rmk140107a}
In light of Propositions~\ref{prop140103a} and~\ref{prop140103ax}, it is natural to ask whether there are other 
standard classes of graphs that have admissible labelings. Some natural candidates can be ruled out
by considering other graphs on five vertices as follows.  All connected graphs with exactly five vertices are shown below.  Proposition~\ref{rmk140107b} can be used to address the disconnected graphs. 

Complete $m$-partite graphs. 
Graphs (9) and (12) from the list below\footnote{Note that the graphs in this list that have admissible labelings are displayed
with one such labeling.  The others, marked NL for ``no label", do not have admissible labelings, as the interested reader is invited to verify.  Graph (6), for example, is addressed in Proposition 5.15.} 
are complete bipartite (namely $K_{1, 4}$ and $K_{2,3}$, respectively) but Graph (9) has an admissible
labeling, while Graph (12) does not.
Graphs (7) and (16) are complete tri-partite,  but Graph (7) has an admissible
labeling, while Graph (16) does not.
Also note that Graph (17) is complete 4-partite, and does not have an admissible labeling.  

Chordal graphs.  Graph (6) is not chordal, while Graph (16) is chordal; neither of these graphs have
admissible labelings. 
Graph (4)  is not chordal, while Graph (3) is chordal; both of these graphs have
admissible labelings. 
\end{remark}

\vskip-1in


\begin{align*}
\begin{pgfpicture}{0cm}{0cm}{0cm}{0cm}
\pgfnodecircle{Node1}[fill]{\pgfxy(-5,1.5)}{0.1cm}
\pgfnodecircle{Node2}[fill]{\pgfxy(-3.57342, 0.463526)}{0.1cm}
\pgfnodecircle{Node3}[fill]{\pgfxy(-4.118322, -1.21352)}{0.1cm}
\pgfnodecircle{Node4}[fill]{\pgfxy(-5.881678, -1.21352)}{0.1cm}
\pgfnodecircle{Node5}[fill]{\pgfxy(-6.42658, 0.463526)}{0.1cm}
\pgfnodeconnline{Node1}{Node5}
\pgfnodeconnline{Node2}{Node3}
\pgfnodeconnline{Node3}{Node4}
\pgfnodeconnline{Node4}{Node5}
\pgfnodecircle{Node6}[fill]{\pgfxy(0,1.5)}{0.1cm}
\pgfnodecircle{Node7}[fill]{\pgfxy(1.42658, 0.463526)}{0.1cm}
\pgfnodecircle{Node8}[fill]{\pgfxy(.881678, -1.21352)}{0.1cm}
\pgfnodecircle{Node9}[fill]{\pgfxy(-.881678, -1.21352)}{0.1cm}
\pgfnodecircle{Node10}[fill]{\pgfxy(-1.42658, 0.463526)}{0.1cm}
\pgfnodeconnline{Node6}{Node7}
\pgfnodeconnline{Node6}{Node10}
\pgfnodeconnline{Node7}{Node8}
\pgfnodeconnline{Node8}{Node9}
\pgfnodeconnline{Node9}{Node10}
\pgfnodecircle{Node11}[fill]{\pgfxy(5,1.5)}{0.1cm}
\pgfnodecircle{Node12}[fill]{\pgfxy(6.42658, 0.463526)}{0.1cm}
\pgfnodecircle{Node13}[fill]{\pgfxy(5.881678, -1.21352)}{0.1cm}
\pgfnodecircle{Node14}[fill]{\pgfxy(4.118322, -1.21352)}{0.1cm}
\pgfnodecircle{Node15}[fill]{\pgfxy(3.57342, 0.463526)}{0.1cm}
\pgfnodeconnline{Node11}{Node12}
\pgfnodeconnline{Node11}{Node15}
\pgfnodeconnline{Node12}{Node13}
\pgfnodeconnline{Node13}{Node14}
\pgfnodeconnline{Node14}{Node15}
\pgfnodeconnline{Node11}{Node13}
\pgfnodeconnline{Node11}{Node14}
\pgfnodeconnline{Node11}{Node15}
\pgfnodeconnline{Node12}{Node14}
\pgfnodeconnline{Node12}{Node15}
\pgfnodeconnline{Node13}{Node15}
\pgfputat{\pgfxy(-5.2, 1.8)}{\pgfbox[left,center]{12}}
\pgfputat{\pgfxy(-3.45, .7)}{\pgfbox[left,center]{56}}
\pgfputat{\pgfxy(-4.1, -1.5)}{\pgfbox[left,center]{45}}
\pgfputat{\pgfxy(-6.3,-1.5)}{\pgfbox[left,center]{34}}
\pgfputat{\pgfxy(-6.9,.7)}{\pgfbox[left,center]{23}}
\pgfputat{\pgfxy(-.2, 1.8)}{\pgfbox[left,center]{12}}
\pgfputat{\pgfxy(1.55, .7)}{\pgfbox[left,center]{15}}
\pgfputat{\pgfxy(.9, -1.5)}{\pgfbox[left,center]{45}}
\pgfputat{\pgfxy(-1.3,-1.5)}{\pgfbox[left,center]{34}}
\pgfputat{\pgfxy(-1.9,.7)}{\pgfbox[left,center]{23}}
\pgfputat{\pgfxy(4.9, 1.8)}{\pgfbox[left,center]{1}}
\pgfputat{\pgfxy(6.55, .7)}{\pgfbox[left,center]{5}}
\pgfputat{\pgfxy(5.9, -1.5)}{\pgfbox[left,center]{4}}
\pgfputat{\pgfxy(3.85,-1.5)}{\pgfbox[left,center]{3}}
\pgfputat{\pgfxy(3.25,.7)}{\pgfbox[left,center]{2}}
\pgfputat{\pgfxy(-5.2,-2.5)}{\pgfbox[left,center]{(1)}}
\pgfputat{\pgfxy(-0.2,-2.5)}{\pgfbox[left,center]{(2)}}
\pgfputat{\pgfxy(4.8,-2.5)}{\pgfbox[left,center]{(3)}}
\end{pgfpicture}
\end{align*}

\vskip-.3in


\begin{align*}
\begin{pgfpicture}{0cm}{0cm}{0cm}{0cm}
\pgfnodecircle{Node1}[fill]{\pgfxy(-5,1.5)}{0.1cm}
\pgfnodecircle{Node2}[fill]{\pgfxy(-3.57342, 0.463526)}{0.1cm}
\pgfnodecircle{Node3}[fill]{\pgfxy(-4.118322, -1.21352)}{0.1cm}
\pgfnodecircle{Node4}[fill]{\pgfxy(-5.881678, -1.21352)}{0.1cm}
\pgfnodecircle{Node5}[fill]{\pgfxy(-6.42658, 0.463526)}{0.1cm}
\pgfnodeconnline{Node1}{Node5}
\pgfnodeconnline{Node2}{Node3}
\pgfnodeconnline{Node3}{Node4}
\pgfnodeconnline{Node4}{Node5}
\pgfnodeconnline{Node1}{Node2}
\pgfnodeconnline{Node2}{Node5}
\pgfnodecircle{Node6}[fill]{\pgfxy(0,1.5)}{0.1cm}
\pgfnodecircle{Node7}[fill]{\pgfxy(1.42658, 0.463526)}{0.1cm}
\pgfnodecircle{Node8}[fill]{\pgfxy(.881678, -1.21352)}{0.1cm}
\pgfnodecircle{Node9}[fill]{\pgfxy(-.881678, -1.21352)}{0.1cm}
\pgfnodecircle{Node10}[fill]{\pgfxy(-1.42658, 0.463526)}{0.1cm}
\pgfnodeconnline{Node6}{Node7}
\pgfnodeconnline{Node6}{Node10}
\pgfnodeconnline{Node7}{Node8}
\pgfnodeconnline{Node8}{Node9}
\pgfnodeconnline{Node9}{Node10}
\pgfnodeconnline{Node7}{Node10}
\pgfnodeconnline{Node7}{Node9}
\pgfnodecircle{Node11}[fill]{\pgfxy(5,1.5)}{0.1cm}
\pgfnodecircle{Node12}[fill]{\pgfxy(6.42658, 0.463526)}{0.1cm}
\pgfnodecircle{Node13}[fill]{\pgfxy(5.881678, -1.21352)}{0.1cm}
\pgfnodecircle{Node14}[fill]{\pgfxy(4.118322, -1.21352)}{0.1cm}
\pgfnodecircle{Node15}[fill]{\pgfxy(3.57342, 0.463526)}{0.1cm}
\pgfnodeconnline{Node11}{Node12}
\pgfnodeconnline{Node11}{Node15}
\pgfnodeconnline{Node12}{Node13}
\pgfnodeconnline{Node13}{Node14}
\pgfnodeconnline{Node14}{Node15}
\pgfnodeconnline{Node11}{Node13}
\pgfnodeconnline{Node12}{Node15}
\pgfputat{\pgfxy(-5.2, 1.8)}{\pgfbox[left,center]{15}}
\pgfputat{\pgfxy(-3.45, .7)}{\pgfbox[left,center]{12}}
\pgfputat{\pgfxy(-4.1, -1.5)}{\pgfbox[left,center]{23}}
\pgfputat{\pgfxy(-6.3,-1.5)}{\pgfbox[left,center]{34}}
\pgfputat{\pgfxy(-6.9,.7)}{\pgfbox[left,center]{14}}
\pgfputat{\pgfxy(-.2, 1.8)}{\pgfbox[left,center]{25}}
\pgfputat{\pgfxy(1.55, .7)}{\pgfbox[left,center]{12}}
\pgfputat{\pgfxy(.9, -1.5)}{\pgfbox[left,center]{13}}
\pgfputat{\pgfxy(-1.3,-1.5)}{\pgfbox[left,center]{14}}
\pgfputat{\pgfxy(-1.9,.7)}{\pgfbox[left,center]{24}}
\pgfputat{\pgfxy(-5.2,-2.5)}{\pgfbox[left,center]{(4)}}
\pgfputat{\pgfxy(-0.2,-2.5)}{\pgfbox[left,center]{(5)}}
\pgfputat{\pgfxy(4.6,-2.5)}{\pgfbox[left,center]{(6: NL)}}
\end{pgfpicture}
\end{align*}

\vskip-.3in


\begin{align*}
\begin{pgfpicture}{0cm}{0cm}{0cm}{0cm}
\pgfnodecircle{Node1}[fill]{\pgfxy(-5,1.5)}{0.1cm}
\pgfnodecircle{Node2}[fill]{\pgfxy(-3.57342, 0.463526)}{0.1cm}
\pgfnodecircle{Node3}[fill]{\pgfxy(-4.118322, -1.21352)}{0.1cm}
\pgfnodecircle{Node4}[fill]{\pgfxy(-5.881678, -1.21352)}{0.1cm}
\pgfnodecircle{Node5}[fill]{\pgfxy(-6.42658, 0.463526)}{0.1cm}
\pgfnodeconnline{Node1}{Node5}
\pgfnodeconnline{Node2}{Node3}
\pgfnodeconnline{Node3}{Node4}
\pgfnodeconnline{Node4}{Node5}
\pgfnodeconnline{Node1}{Node2}
\pgfnodeconnline{Node2}{Node5}
\pgfnodeconnline{Node1}{Node3}
\pgfnodeconnline{Node1}{Node4}
\pgfnodecircle{Node6}[fill]{\pgfxy(0,1.5)}{0.1cm}
\pgfnodecircle{Node7}[fill]{\pgfxy(1.42658, 0.463526)}{0.1cm}
\pgfnodecircle{Node8}[fill]{\pgfxy(.881678, -1.21352)}{0.1cm}
\pgfnodecircle{Node9}[fill]{\pgfxy(-.881678, -1.21352)}{0.1cm}
\pgfnodecircle{Node10}[fill]{\pgfxy(-1.42658, 0.463526)}{0.1cm}
\pgfnodeconnline{Node6}{Node7}
\pgfnodeconnline{Node6}{Node10}
\pgfnodeconnline{Node8}{Node9}
\pgfnodeconnline{Node9}{Node10}
\pgfnodeconnline{Node8}{Node10}
\pgfnodecircle{Node11}[fill]{\pgfxy(5,1.5)}{0.1cm}
\pgfnodecircle{Node12}[fill]{\pgfxy(6.42658, 0.463526)}{0.1cm}
\pgfnodecircle{Node13}[fill]{\pgfxy(5.881678, -1.21352)}{0.1cm}
\pgfnodecircle{Node14}[fill]{\pgfxy(4.118322, -1.21352)}{0.1cm}
\pgfnodecircle{Node15}[fill]{\pgfxy(3.57342, 0.463526)}{0.1cm}
\pgfnodeconnline{Node11}{Node12}
\pgfnodeconnline{Node11}{Node15}
\pgfnodeconnline{Node11}{Node13}
\pgfnodeconnline{Node11}{Node14}
\pgfputat{\pgfxy(-5.2, 1.8)}{\pgfbox[left,center]{12}}
\pgfputat{\pgfxy(-3.45, .7)}{\pgfbox[left,center]{13}}
\pgfputat{\pgfxy(-4.1, -1.5)}{\pgfbox[left,center]{23}}
\pgfputat{\pgfxy(-6.3,-1.5)}{\pgfbox[left,center]{24}}
\pgfputat{\pgfxy(-6.9,.7)}{\pgfbox[left,center]{14}}
\pgfputat{\pgfxy(-.2, 1.8)}{\pgfbox[left,center]{45}}
\pgfputat{\pgfxy(1.55, .7)}{\pgfbox[left,center]{56}}
\pgfputat{\pgfxy(.9, -1.5)}{\pgfbox[left,center]{12}}
\pgfputat{\pgfxy(-1.3,-1.5)}{\pgfbox[left,center]{13}}
\pgfputat{\pgfxy(-1.9,.7)}{\pgfbox[left,center]{14}}
\pgfputat{\pgfxy(4.65, 1.8)}{\pgfbox[left,center]{1234}}
\pgfputat{\pgfxy(6.55, .7)}{\pgfbox[left,center]{2348}}
\pgfputat{\pgfxy(5.9, -1.5)}{\pgfbox[left,center]{1247}}
\pgfputat{\pgfxy(3.7,-1.5)}{\pgfbox[left,center]{1346}}
\pgfputat{\pgfxy(2.8,.7)}{\pgfbox[left,center]{1235}}
\pgfputat{\pgfxy(-5.2,-2.5)}{\pgfbox[left,center]{(7)}}
\pgfputat{\pgfxy(-0.2,-2.5)}{\pgfbox[left,center]{(8)}}
\pgfputat{\pgfxy(4.8,-2.5)}{\pgfbox[left,center]{(9)}}
\end{pgfpicture}
\end{align*}

\vskip-.3in


\begin{align*}
\begin{pgfpicture}{0cm}{0cm}{0cm}{-.5cm}
\pgfnodecircle{Node1}[fill]{\pgfxy(-5,1.5)}{0.1cm}
\pgfnodecircle{Node2}[fill]{\pgfxy(-3.57342, 0.463526)}{0.1cm}
\pgfnodecircle{Node3}[fill]{\pgfxy(-4.118322, -1.21352)}{0.1cm}
\pgfnodecircle{Node4}[fill]{\pgfxy(-5.881678, -1.21352)}{0.1cm}
\pgfnodecircle{Node5}[fill]{\pgfxy(-6.42658, 0.463526)}{0.1cm}
\pgfnodeconnline{Node1}{Node5}
\pgfnodeconnline{Node2}{Node3}
\pgfnodeconnline{Node3}{Node4}
\pgfnodeconnline{Node4}{Node5}
\pgfnodeconnline{Node2}{Node5}
\pgfnodecircle{Node6}[fill]{\pgfxy(0,1.5)}{0.1cm}
\pgfnodecircle{Node7}[fill]{\pgfxy(1.42658, 0.463526)}{0.1cm}
\pgfnodecircle{Node8}[fill]{\pgfxy(.881678, -1.21352)}{0.1cm}
\pgfnodecircle{Node9}[fill]{\pgfxy(-.881678, -1.21352)}{0.1cm}
\pgfnodecircle{Node10}[fill]{\pgfxy(-1.42658, 0.463526)}{0.1cm}
\pgfnodeconnline{Node7}{Node8}
\pgfnodeconnline{Node8}{Node9}
\pgfnodeconnline{Node9}{Node10}
\pgfnodeconnline{Node6}{Node9}
\pgfnodecircle{Node11}[fill]{\pgfxy(5,1.5)}{0.1cm}
\pgfnodecircle{Node12}[fill]{\pgfxy(6.42658, 0.463526)}{0.1cm}
\pgfnodecircle{Node13}[fill]{\pgfxy(5.881678, -1.21352)}{0.1cm}
\pgfnodecircle{Node14}[fill]{\pgfxy(4.118322, -1.21352)}{0.1cm}
\pgfnodecircle{Node15}[fill]{\pgfxy(3.57342, 0.463526)}{0.1cm}
\pgfnodeconnline{Node11}{Node12}
\pgfnodeconnline{Node11}{Node15}
\pgfnodeconnline{Node12}{Node13}
\pgfnodeconnline{Node13}{Node14}
\pgfnodeconnline{Node14}{Node15}
\pgfnodeconnline{Node13}{Node15}
\pgfnodeconnline{Node12}{Node14}
\pgfputat{\pgfxy(-5.2, 1.8)}{\pgfbox[left,center]{156}}
\pgfputat{\pgfxy(-3.45, .7)}{\pgfbox[left,center]{123}}
\pgfputat{\pgfxy(-4.1, -1.5)}{\pgfbox[left,center]{234}}
\pgfputat{\pgfxy(-6.3,-1.5)}{\pgfbox[left,center]{345}}
\pgfputat{\pgfxy(-6.9,.7)}{\pgfbox[left,center]{135}}
\pgfputat{\pgfxy(-.2, 1.8)}{\pgfbox[left,center]{137}}
\pgfputat{\pgfxy(1.55, .7)}{\pgfbox[left,center]{345}}
\pgfputat{\pgfxy(.9, -1.5)}{\pgfbox[left,center]{234}}
\pgfputat{\pgfxy(-1.3,-1.5)}{\pgfbox[left,center]{123}}
\pgfputat{\pgfxy(-1.9,.7)}{\pgfbox[left,center]{126}}
\pgfputat{\pgfxy(-5.2,-2.5)}{\pgfbox[left,center]{(10)}}
\pgfputat{\pgfxy(-0.2,-2.5)}{\pgfbox[left,center]{(11)}}
\pgfputat{\pgfxy(4.6,-2.5)}{\pgfbox[left,center]{(12: NL)}}
\end{pgfpicture}
\end{align*}

\pagebreak


\begin{align*}
\begin{pgfpicture}{0cm}{0cm}{0cm}{2cm}
\pgfnodecircle{Node1}[fill]{\pgfxy(-5,1.5)}{0.1cm}
\pgfnodecircle{Node2}[fill]{\pgfxy(-3.57342, 0.463526)}{0.1cm}
\pgfnodecircle{Node3}[fill]{\pgfxy(-4.118322, -1.21352)}{0.1cm}
\pgfnodecircle{Node4}[fill]{\pgfxy(-5.881678, -1.21352)}{0.1cm}
\pgfnodecircle{Node5}[fill]{\pgfxy(-6.42658, 0.463526)}{0.1cm}
\pgfnodeconnline{Node1}{Node5}
\pgfnodeconnline{Node2}{Node3}
\pgfnodeconnline{Node3}{Node4}
\pgfnodeconnline{Node4}{Node5}
\pgfnodeconnline{Node2}{Node5}
\pgfnodeconnline{Node2}{Node4}
\pgfnodecircle{Node6}[fill]{\pgfxy(0,1.5)}{0.1cm}
\pgfnodecircle{Node7}[fill]{\pgfxy(1.42658, 0.463526)}{0.1cm}
\pgfnodecircle{Node8}[fill]{\pgfxy(.881678, -1.21352)}{0.1cm}
\pgfnodecircle{Node9}[fill]{\pgfxy(-.881678, -1.21352)}{0.1cm}
\pgfnodecircle{Node10}[fill]{\pgfxy(-1.42658, 0.463526)}{0.1cm}
\pgfnodeconnline{Node6}{Node10}
\pgfnodeconnline{Node7}{Node8}
\pgfnodeconnline{Node7}{Node10}
\pgfnodeconnline{Node7}{Node9}
\pgfnodeconnline{Node8}{Node9}
\pgfnodeconnline{Node9}{Node10}
\pgfnodeconnline{Node8}{Node10}
\pgfnodecircle{Node11}[fill]{\pgfxy(5,1.5)}{0.1cm}
\pgfnodecircle{Node12}[fill]{\pgfxy(6.42658, 0.463526)}{0.1cm}
\pgfnodecircle{Node13}[fill]{\pgfxy(5.881678, -1.21352)}{0.1cm}
\pgfnodecircle{Node14}[fill]{\pgfxy(4.118322, -1.21352)}{0.1cm}
\pgfnodecircle{Node15}[fill]{\pgfxy(3.57342, 0.463526)}{0.1cm}
\pgfnodeconnline{Node11}{Node12}
\pgfnodeconnline{Node11}{Node15}
\pgfnodeconnline{Node12}{Node13}
\pgfnodeconnline{Node12}{Node14}
\pgfnodeconnline{Node12}{Node15}
\pgfnodeconnline{Node13}{Node14}
\pgfnodeconnline{Node13}{Node15}
\pgfnodeconnline{Node14}{Node15}
\pgfputat{\pgfxy(-5.25, 1.8)}{\pgfbox[left,center]{356}}
\pgfputat{\pgfxy(-3.45, .7)}{\pgfbox[left,center]{123}}
\pgfputat{\pgfxy(-4.1, -1.5)}{\pgfbox[left,center]{124}}
\pgfputat{\pgfxy(-6.3,-1.5)}{\pgfbox[left,center]{125}}
\pgfputat{\pgfxy(-7.05,.7)}{\pgfbox[left,center]{235}}
\pgfputat{\pgfxy(-.2, 1.8)}{\pgfbox[left,center]{56}}
\pgfputat{\pgfxy(1.55, .7)}{\pgfbox[left,center]{12}}
\pgfputat{\pgfxy(.9, -1.5)}{\pgfbox[left,center]{13}}
\pgfputat{\pgfxy(-1.3,-1.5)}{\pgfbox[left,center]{14}}
\pgfputat{\pgfxy(-1.9,.7)}{\pgfbox[left,center]{15}}
\pgfputat{\pgfxy(4.8, 1.8)}{\pgfbox[left,center]{25}}
\pgfputat{\pgfxy(6.55, .7)}{\pgfbox[left,center]{12}}
\pgfputat{\pgfxy(5.9, -1.5)}{\pgfbox[left,center]{13}}
\pgfputat{\pgfxy(3.7,-1.5)}{\pgfbox[left,center]{14}}
\pgfputat{\pgfxy(3.1,.7)}{\pgfbox[left,center]{15}}
\pgfputat{\pgfxy(-5.2,-2.5)}{\pgfbox[left,center]{(13)}}
\pgfputat{\pgfxy(-0.2,-2.5)}{\pgfbox[left,center]{(14)}}
\pgfputat{\pgfxy(4.8,-2.5)}{\pgfbox[left,center]{(15)}}
\end{pgfpicture}
\end{align*}

\vskip-.3in


\begin{align*}
\begin{pgfpicture}{0cm}{0cm}{0cm}{5cm}
\pgfnodecircle{Node1}[fill]{\pgfxy(-5,1.5)}{0.1cm}
\pgfnodecircle{Node2}[fill]{\pgfxy(-3.57342, 0.463526)}{0.1cm}
\pgfnodecircle{Node3}[fill]{\pgfxy(-4.118322, -1.21352)}{0.1cm}
\pgfnodecircle{Node4}[fill]{\pgfxy(-5.881678, -1.21352)}{0.1cm}
\pgfnodecircle{Node5}[fill]{\pgfxy(-6.42658, 0.463526)}{0.1cm}
\pgfnodeconnline{Node1}{Node5}
\pgfnodeconnline{Node1}{Node2}
\pgfnodeconnline{Node2}{Node3}
\pgfnodeconnline{Node3}{Node5}
\pgfnodeconnline{Node4}{Node5}
\pgfnodeconnline{Node2}{Node5}
\pgfnodeconnline{Node2}{Node4}
\pgfnodecircle{Node6}[fill]{\pgfxy(0,1.5)}{0.1cm}
\pgfnodecircle{Node7}[fill]{\pgfxy(1.42658, 0.463526)}{0.1cm}
\pgfnodecircle{Node8}[fill]{\pgfxy(.881678, -1.21352)}{0.1cm}
\pgfnodecircle{Node9}[fill]{\pgfxy(-.881678, -1.21352)}{0.1cm}
\pgfnodecircle{Node10}[fill]{\pgfxy(-1.42658, 0.463526)}{0.1cm}
\pgfnodeconnline{Node6}{Node7}
\pgfnodeconnline{Node6}{Node8}
\pgfnodeconnline{Node6}{Node9}
\pgfnodeconnline{Node6}{Node10}
\pgfnodeconnline{Node7}{Node8}
\pgfnodeconnline{Node7}{Node9}
\pgfnodeconnline{Node8}{Node9}
\pgfnodeconnline{Node9}{Node10}
\pgfnodeconnline{Node8}{Node10}
\pgfnodecircle{Node11}[fill]{\pgfxy(5,1.5)}{0.1cm}
\pgfnodecircle{Node12}[fill]{\pgfxy(6.42658, 0.463526)}{0.1cm}
\pgfnodecircle{Node13}[fill]{\pgfxy(5.881678, -1.21352)}{0.1cm}
\pgfnodecircle{Node14}[fill]{\pgfxy(4.118322, -1.21352)}{0.1cm}
\pgfnodecircle{Node15}[fill]{\pgfxy(3.57342, 0.463526)}{0.1cm}
\pgfnodeconnline{Node11}{Node15}
\pgfnodeconnline{Node12}{Node13}
\pgfnodeconnline{Node13}{Node14}
\pgfnodeconnline{Node13}{Node15}
\pgfnodeconnline{Node14}{Node15}
\pgfputat{\pgfxy(4.75, 1.8)}{\pgfbox[left,center]{125}}
\pgfputat{\pgfxy(6.55, .7)}{\pgfbox[left,center]{246}}
\pgfputat{\pgfxy(5.9, -1.5)}{\pgfbox[left,center]{234}}
\pgfputat{\pgfxy(3.7,-1.5)}{\pgfbox[left,center]{134}}
\pgfputat{\pgfxy(2.95,.7)}{\pgfbox[left,center]{123}}
\pgfputat{\pgfxy(-5.6,-2.5)}{\pgfbox[left,center]{(16: NL)}}
\pgfputat{\pgfxy(-.6,-2.5)}{\pgfbox[left,center]{(17: NL)}}
\pgfputat{\pgfxy(4.8,-2.5)}{\pgfbox[left,center]{(18)}}
\end{pgfpicture}
\end{align*}

\vskip-.3in


\begin{align*}
\begin{pgfpicture}{0cm}{0cm}{0cm}{5cm}
\pgfnodecircle{Node1}[fill]{\pgfxy(-5,1.5)}{0.1cm}
\pgfnodecircle{Node2}[fill]{\pgfxy(-3.57342, 0.463526)}{0.1cm}
\pgfnodecircle{Node3}[fill]{\pgfxy(-4.118322, -1.21352)}{0.1cm}
\pgfnodecircle{Node4}[fill]{\pgfxy(-5.881678, -1.21352)}{0.1cm}
\pgfnodecircle{Node5}[fill]{\pgfxy(-6.42658, 0.463526)}{0.1cm}
\pgfnodeconnline{Node1}{Node2}
\pgfnodeconnline{Node2}{Node3}
\pgfnodeconnline{Node3}{Node4}
\pgfnodeconnline{Node4}{Node5}
\pgfnodeconnline{Node2}{Node5}
\pgfnodeconnline{Node2}{Node4}
\pgfnodecircle{Node6}[fill]{\pgfxy(0,1.5)}{0.1cm}
\pgfnodecircle{Node7}[fill]{\pgfxy(1.42658, 0.463526)}{0.1cm}
\pgfnodecircle{Node8}[fill]{\pgfxy(.881678, -1.21352)}{0.1cm}
\pgfnodecircle{Node9}[fill]{\pgfxy(-.881678, -1.21352)}{0.1cm}
\pgfnodecircle{Node10}[fill]{\pgfxy(-1.42658, 0.463526)}{0.1cm}
\pgfnodeconnline{Node6}{Node7}
\pgfnodeconnline{Node7}{Node8}
\pgfnodeconnline{Node7}{Node9}
\pgfnodeconnline{Node7}{Node10}
\pgfnodeconnline{Node8}{Node9}
\pgfnodecircle{Node11}[fill]{\pgfxy(5,1.5)}{0.1cm}
\pgfnodecircle{Node12}[fill]{\pgfxy(6.42658, 0.463526)}{0.1cm}
\pgfnodecircle{Node13}[fill]{\pgfxy(5.881678, -1.21352)}{0.1cm}
\pgfnodecircle{Node14}[fill]{\pgfxy(4.118322, -1.21352)}{0.1cm}
\pgfnodecircle{Node15}[fill]{\pgfxy(3.57342, 0.463526)}{0.1cm}
\pgfnodeconnline{Node11}{Node12}
\pgfnodeconnline{Node11}{Node15}
\pgfnodeconnline{Node12}{Node13}
\pgfnodeconnline{Node12}{Node14}
\pgfnodeconnline{Node12}{Node15}
\pgfnodeconnline{Node13}{Node14}
\pgfputat{\pgfxy(-5.25, 1.8)}{\pgfbox[left,center]{136}}
\pgfputat{\pgfxy(-3.45, .7)}{\pgfbox[left,center]{123}}
\pgfputat{\pgfxy(-4.1, -1.5)}{\pgfbox[left,center]{124}}
\pgfputat{\pgfxy(-6.3,-1.5)}{\pgfbox[left,center]{125}}
\pgfputat{\pgfxy(-7.05,.7)}{\pgfbox[left,center]{235}}
\pgfputat{\pgfxy(-.25, 1.8)}{\pgfbox[left,center]{125}}
\pgfputat{\pgfxy(1.55, .7)}{\pgfbox[left,center]{123}}
\pgfputat{\pgfxy(.9, -1.5)}{\pgfbox[left,center]{234}}
\pgfputat{\pgfxy(-1.3,-1.5)}{\pgfbox[left,center]{134}}
\pgfputat{\pgfxy(-2.1,.7)}{\pgfbox[left,center]{126}}
\pgfputat{\pgfxy(4.6, 1.8)}{\pgfbox[left,center]{1235}}
\pgfputat{\pgfxy(6.55, .7)}{\pgfbox[left,center]{1234}}
\pgfputat{\pgfxy(5.9, -1.5)}{\pgfbox[left,center]{2346}}
\pgfputat{\pgfxy(3.5,-1.5)}{\pgfbox[left,center]{1246}}
\pgfputat{\pgfxy(2.8,.7)}{\pgfbox[left,center]{1345}}
\pgfputat{\pgfxy(-5.2,-2.5)}{\pgfbox[left,center]{(18)}}
\pgfputat{\pgfxy(-0.2,-2.5)}{\pgfbox[left,center]{(20)}}
\pgfputat{\pgfxy(4.8,-2.5)}{\pgfbox[left,center]{(21)}}
\end{pgfpicture}
\end{align*}

\vskip1.5in

\begin{Question}
Can regular graphs always be labeled?
\end{Question}

\begin{Problem}
Characterize the graphs that have admissible labelings.
\end{Problem}

Our final result shows that the question of admissible labelings for graphs in general boils down to the connected case.

\begin{proposition}\label{rmk140107b}
Let $G$ be a graph with connected components $G_1,\ldots,G_t$. Then $G$ has an admissible labeling
if and only if each $G_i$ has an admissible labeling.
\end{proposition}

\begin{proof}
If $G$ has an admissible labeling, then so does each $G_i$, being an induced subgraph by Lemma~\ref{lem140103b}(1).

Conversely, assume that for $i=1,\ldots,t$ the component $G_i$ has an
admissible labeling $\phi_i\colon V_i\into \binom{[n_i]}{s_i}$.
Set $s:=\max\{s_1,\ldots,s_t\}$.
The proof of Corollary~\ref{cor140103a} shows that we may assume that $s_i=s\geq 2$ for each $i$.
Set $n=n_1+\cdots+n_t$ and define $\phi\colon V\into \binom{[n]}{s}$ as follows.
Each vertex $v$ is in a unique $V_i$, say with
$\phi_i(v)=\{a_1,\ldots,a_s\}$. 
Set $m_i:=\sum_{j=1}^{i-1}n_j$.
Then we set
$\phi(v)=\{a_1+m_i,\ldots,a_s+m_i\}$.
Notice that we have $m_i<a_p+m_i\leq m_i+n_i=m_{i+1}$ for each $i$.
It follows that, for $v\in V_i$ and $w\in V_j$ with $i\neq j$, we have $\phi(v)\cap\phi(w)=\emptyset$. 
In particular, $\phi$ satisfies condition (2) from Definition~\ref{labels}
for the non-adjacent vertices $v$ and $w$. It is straightforward to show that
$\phi$ satisfies the remaining conditions of 
Definition~\ref{labels} as well.
\end{proof}

\section*{Acknowledgments}
We are grateful to Graham Leuschke, Warren Shreve, and Jessica Striker for useful conversations about this material.


\providecommand{\bysame}{\leavevmode\hbox to3em{\hrulefill}\thinspace}
\providecommand{\MR}{\relax\ifhmode\unskip\space\fi MR }
\providecommand{\MRhref}[2]{%
  \href{http://www.ams.org/mathscinet-getitem?mr=#1}{#2}
}
\providecommand{\href}[2]{#2}

\end{document}